\documentclass[11pt]{article}
\usepackage{epsf}
\usepackage{color}
\usepackage{epsfig}

\usepackage{graphicx}
\usepackage{amsmath}
\usepackage{latexsym}
\usepackage{amssymb}
\usepackage{amsfonts}

\usepackage{booktabs}
\usepackage{tabularx}
\usepackage{array}

\usepackage{listings}
\usepackage{hyperref}
\usepackage[T1]{fontenc}
\usepackage{lmodern}
\usepackage{xspace}

\usepackage{tikz}
\usepackage{pgfplots}
\usepackage[many]{tcolorbox}

\tikzstyle{tituloStyle} = [
draw=green!30!black,
fill=green!9,
very thick,
rectangle,
rounded corners=3mm,
inner sep=10pt,
inner ysep=15pt,
anchor=west,
opacity=0.9,
]

\colorlet{fondo}{blue!10!white}
\colorlet{coments}{blue!40!white}	
\pgfplotsset{every axis/.append style={
	axis x line=middle,    
	axis y line=middle,    
	axis line style={<->,color=blue}, 
	xlabel={$x$},          
	ylabel={$y$},          
	}}

\usepackage{a4}

\newtheorem{theorem}{Theorem}

\newtheorem{proposition}[theorem]{Proposition}

\newtheorem{corollary}[theorem]{Corollary}
\newenvironment{proof}[1][Proof]{\textbf{#1.} }
     {\    \rule{0.5em}{0.5em}}

\def\beq{\begin{equation}}
\def\eeq{\end{equation}}
\def\beqnn{\begin{eqnarray*}}
\def\eeqnn{\end{eqnarray*}}
\def\bea{\begin{eqnarray}}
\def\eea{\end{eqnarray}}

\newcommand{\Cay}{\mathrm{Cay}}

\newcommand{\e}{\ensuremath{\mathrm{e}}}

\newcommand{\R}{\mathbb{R}}

\def\theequation{\arabic{section}.\arabic{equation}}

\begin{document}

\title{On a class of modified Cayley--Magnus methods}

\author{Sergio Blanes
	\thanks{Instituto Universitario de Matem\'{a}tica Multidisciplinar, 
		Universitat Polit\`{e}cnica de Val\`{e}ncia. Spain. Email: \texttt{serblaza@imm.upv.es}}
   \and
 Fernando Casas 
 \thanks{Departament de Matem\`atiques and IMAC, Universitat Jaume I, Spain. Email: \texttt{casas@uji.es}}
   \and
Arieh Iserles 
\thanks{Department of Applied Mathematics and Theoretical Physics,
	University of Cambridge, UK. Email: \texttt{ai10@cam.ac.uk}}
   }

\maketitle


\begin{abstract}

We introduce a new class of numerical integrators for the time integration of non-autonomous linear ordinary differential equations whose coefficient matrix is sparse and evolves within a quadratic matrix Lie group. In contrast to standard Lie group integrators, the proposed methods avoid the evaluation of matrix exponentials acting on vectors and instead rely on solving a sequence of linear systems with sparse coefficient matrices. Moreover, they are well suited for problems arising from unbounded operators, as they inherently produce bounded solutions. We construct optimised schemes of orders four and six and assess their performance on a representative numerical example, demonstrating clear advantages over existing Lie-group integrators.

%

\end{abstract}

\section{Introduction}  

\paragraph{The class of problems.} In this work we are concerned with the numerical time integration of ordinary differential equations (ODEs) resulting from 
the space discretization of linear partial differential equations (PDEs) of the form
\begin{equation}\label{eq:2}
	\frac{\partial \psi(x,t)}{\partial t}={\cal L}(x,t)\psi(x,t), \qquad \psi(x,0)=\psi_0(x),\quad t\geq0,
\end{equation}
where $\psi: \mathbb{R}^d \times \mathbb{R}_+ \longmapsto \mathbb{C}$, ${\cal L}(x,t)$ is a (possibly unbounded) linear operator.
When finite difference, Galerkin or spectral methods are used for the discretization of \eqref{eq:2} in space, a typical outcome is a 
non-autonomous linear ODE
\begin{equation}\label{eq:1}
	x'=A(t)x, \qquad x(0)=x_0\in\mathbb{C}^{N},
\end{equation}
where $A(t) \in \mathrm{GL}_N(\mathbb{C})$, the group of all $N \times N$ non-singular complex-valued matrices, the dimension  $N$ is usually quite large and $A$ has some specific structure. 
For instance, suppose that there exists 
an orthonormal basis of functions $\Phi=\{\phi_n(x)\}_{n\in\mathbb{Z}_+}$ so that the solution of \eqref{eq:2} can be written in the form
\[
  \psi(x,t) = \sum_{j=0}^{\infty} x_j(t) \phi_j(x).
\]
If this series is truncated after, say, $N$ terms, then the coefficients $x_j(t)$ satisfy the equation \eqref{eq:1} with a matrix $A$ with entries
$A_{i j}=\left\langle{\cal L}\phi_j,\phi_i\right\rangle$. Once the basis is appropriately chosen, the matrix $A$ is sparse and with a relatively simple structure
\cite{iserles24mfo,iserles25ssm}. As a consequence, the action of $A$, or more generally of a function of $A$ on any vector
$v \in \mathbb{C}^N$, can be  computed efficiently.

We are mainly interested in the case where, in addition, $A(t)$ takes values in the Lie algebra 
\begin{equation} \label{alg1}
  o_N(\mathbb{C},J) = \{ \Omega \in \mathbb{C}^{N \times N} \, : \,  \Omega^* J + J \Omega = 0 \},
\end{equation}  
where $J$ is some constant matrix in $\mathrm{GL}_N(\mathbb{C})$. In that case,
the fundamental matrix $X(t)$ of \eqref{eq:1} evolves in the Lie group
\begin{equation} \label{group1}
  \mathrm{O}_N(\mathbb{C},J) = \{ X \in \mathrm{GL}_N(\mathbb{C}) \, : \, X^* J X = J \},
\end{equation}
whose associated Lie algebra is  $o_N(\mathbb{C},J)$. In the real case, \eqref{group1} is called a $J$-orthogonal (or quadratic) Lie group \cite{postnikov94lga}. 
Examples of this class are the unitary and orthogonal group (when $J=I$, the identity matrix), the symplectic group (when $J$ is the basic
symplectic structure matrix) and the Lorentz group (corresponding to $N=4$ and  $J = \mathrm{diag}(1,-1,-1,-1)$).
Another relevant application arises when eq. \eqref{eq:2} is the linear Schr\"odinger equation, in which case the associated ODE \eqref{eq:1} evolves
within the unitary group.

When constructing numerical methods for the time integration of \eqref{eq:1}, two essential criteria must be satisfied: firstly, the scheme should exploit the sparsity of $A$ to achieve computational efficiency; secondly, the numerical approximation should preserve the qualitative features of the exact solution, in particular it should evolve on the Lie group $ \mathrm{O}_N(\mathbb{C},J)$.

\paragraph{Lie-group integrators.}
There are, as is well known, numerous methods for solving \eqref{eq:1} that satisfy one of the above requirements, but in fact only few  satisfy both. An important class
of schemes in this setting is formed by Lie-group integrators: they provide by construction numerical approximations evolving on $ \mathrm{O}_N(\mathbb{C},J)$, thus
ensuring preservation of symmetries and conservation laws inherent to this algebraic structure. Among them, we  mention 
Magnus  integrators \cite{iserles99ots,blanes09tma}, Commutator-Free Quasi-Magnus integrators (CFQM) \cite{blanes06fas,alvermann11hoc,blanes17hoc}, and those based on the Fer expansion \cite{blanes09tma}. 
Each of them has, however, its limitations in this context. Thus, both Magnus and Fer integrators require computing iterated integrals
of nested commutators of $A$ at different times and the sparsity of the matrix $A$ is typically lost. In addition, they must evaluate the action of
one or more exponentials on vectors. In consequence, the previous first criterion is not
satisfied. 

With respect to CFQM methods, although the sparsity of $A$ is retained when building the scheme, they still require computing the action of
exponentials on vectors, so that they are computationally expensive for most problems.

Other numerical schemes exist that preserve the Lie group structure of $ \mathrm{O}_N(\mathbb{C},J)$ without relying on matrix exponentials. A notable example is provided by Cayley methods \cite{diele98tct, iserles01oct, iserles00lgm, iserles00otd, marthinsen01qmb}. However, these methods still involve commutators and products of the matrix $A$, rendering the resulting schemes relatively costly. 

\paragraph{Rational methods.} Suppose that we apply the symmetric second-order implicit midpoint rule to advance the numerical solution of \eqref{eq:1} from 
time $t_n$ to time $t_{n+1} = t_n + h$ with time-step  $h$. The algorithm reads
\[
  x_{n+1}= x_n+ \tfrac{h}{2} A_{1/2}(x_n+x_{n+1} ),
\]
where $A_{1/2} \equiv A(t_n + h/2)$, and it produces an approximation $x_{n+1} \approx x(t_{n+1})$. Equivalently,
\begin{equation} \label{eq:Midpoint}
 x_{n+1}=(I-\tfrac{h}{2} A_{1/2})^{-1}(I+\tfrac{h}2A_{1/2}) \, x_n.
\end{equation}
Since $A_{1/2} \in o_N(\mathbb{C},J)$ and \eqref{eq:Midpoint} is nothing but the Cayley transform \cite{postnikov94lga}, it is clear that the approximation rendered by
this method remains in $ \mathrm{O}_N(\mathbb{C},J)$. In addition, it only requires the solution of the linear system
\begin{displaymath}
	(I-\tfrac{h}2A_{1/2})x_{n+1}=(I+\tfrac{h}2A_{1/2})x_n,
\end{displaymath}
taking full advantage of the sparsity of $A$. Unfortunately, this method is only second-order, and thus not appropriate when high accuracy is desired.
Higher order methods of this form are still able to preserve the Lie group structure of $ \mathrm{O}_N(\mathbb{C},J)$. A case in point is the class of 
implicit Runge--Kutta Gauss--Legendre (RKGL) methods 
\cite{iserles09afc, sanz-serna94nhp}, 
although their formulation falls short of fully exploiting the simple structure of $A$.

\paragraph{Our contribution.} In this work we propose a new class of efficient schemes  that generalise \eqref{eq:Midpoint} to higher order,  whereas
preserving the Lie group structure associated with \eqref{group1}. For one time step $h$, they are of the form $x_{n+1} =  \Psi_h^{[p]} x_n$, with
\begin{equation}\label{eq:Cayley_m}
 \Psi_h^{[p]} =\prod_{k=1}^m\Cay (h \, {\bf u}_k \ast{\bf A}),
\qquad \qquad
\Cay(x)=\frac{1+\frac{1}2x}{1-\frac{1}2x},
\end{equation}
and
\begin{displaymath}
  {\bf u}_k \ast{\bf A}= \sum_{i=1}^{\ell} u_{k,i} \, A_i, \qquad  A_i=A(t_n+c_ih), \qquad i=1,2,\ldots,\ell. 
\end{displaymath}  
Here $c_i$, $i=1,\ldots,\ell$, are the nodes of a chosen quadrature rule, and $u_{k,i}$ are coefficients to be determined, depending on the quadrature nodes, $c_i$, and the number of maps used, $m$. 
Note that the implementation requires sequential solution of linear systems of the form
\begin{displaymath}
  \left( I - \frac{h}{2} {\bf u}_k \ast{\bf A} \right) y_{k}= b_k, \qquad k=1,\ldots,m,
\end{displaymath}
where $b_k$ are known vectors, but the corresponding coefficient matrix remains sparse, since it only involves linear combinations of the $A_i$s.
Thus, the new schemes, which we call \emph{modified Cayley--Magnus methods\/}, can be seen as combining the advantages of commutator-free quasi-Magnus methods while avoiding their main drawbacks. In particular, they do not require the evaluation of matrix exponentials acting on vectors, but instead rely solely on solving a sequence of linear systems with sparse coefficient matrices.
Moreover, these methods are well suited for problems arising from unbounded operators, as they inherently produce bounded solutions. In this sense, they can be interpreted as introducing a form of high-frequency filtering in both oscillatory and stiff problems, while still preserving the quadratic Lie group structure of 
$\mathrm{O}_N(\mathbb{C},J)$.


We should comment that schemes of the form \eqref{eq:Cayley_m} have already been explored in the recent literature. Thus, in \cite{maslovskaya26cfc}, a method of order 4  has been constructed and applied to the linear time-dependent Schr\"odinger equation and in 
quantum optimal control
\cite{wembe26ccf}. Here we carry out a systematic analysis and construct 
optimised methods of orders $p=4$ and $p=6$.

The paper is organized as follows. In Section~\ref{sect.2}, we recall that the solution of \eqref{eq:1} can be approximated up to order $p = 2\ell$ using only the values of the coefficient matrix $A$ evaluated at the $\ell$ nodes of a quadrature rule of order $2\ell$. Although this is a well-known result in the context of Lie group integrators, we include  for completeness a detailed and self-contained proof in the appendix.
In Section~\ref{sect.3}, we introduce the class of modified Cayley--Magnus integrators considered in this work and discuss their relationship with Magnus methods. This connection is then exploited to construct new schemes of orders four and six, which are subsequently tested on a representative example in Section~\ref{sect.4}. 

{
	A major issue with expansions in a Lie algebra is that, due to the anti-symmetry of the commutator and to the Jacobi identity, the expansion possesses a great deal of redundancy. This redundancy can be further increased by an astute choice of a basis, which causes the grade of terms (that is, the power of the step size preceding them) to increase rapidly: clearly, terms of grade exceeding the desired order can be disposed with. Careful exploitation of these features, as well as of time symmetry of underlying approximations, leads to huge reduction in the required number of terms, hence in computational expense of a Magnus expansion \cite{iserles00lgm, munthekaas99cia}. Extending this reasoning to Cayley--Magnus expansion exhibits similarly beneficial effect but is technically more challenging. In Section~\ref{sect.3} we use calculus in a homogeneous algebra, {\em \`a la\/} \cite{iserles00otd}, to derive optimal integrators of this kind.
}

\setcounter{equation}{0}
\section{Approximating the exact solution by quadratures: Magnus integrators}
\label{sect.2}

The formal solution of the initial value problem \eqref{eq:1} at time $t=h$ can be obtained by applying Picard iteration, 
\begin{equation}\label{eq:SolIt}
  x(h)={\cal P}(h)x_0=
 \left(\sum_{k=0}^{\infty} I_k(h) \right)  x_0,
\end{equation}
where $I_0(h)=I$ (the identity matrix), and each term $I_k(h)$, $k \ge 1$, is given by the iterated integral
\begin{equation} \label{eq:ii_A}
  I_k(h)=\int_0^h dt_1\int_0^{t_1} dt_2\cdots \int_0^{t_{k-1}} dt_k \, A(t_1) A(t_2) \cdots A(t_k).
\end{equation}
In practice, however, one has to truncate the infinite series $\mathcal{P}(h)$ after, say, $p$ terms, and compute
\[
  x_{1} = \mathcal{P}^{[p]}(h) \, x_0, \qquad \mbox{ with } \qquad \mathcal{P}^{[p]}(h) = \sum_{k=0}^{p} I_k(h), 
\]
so that $x_{1} \approx x(h) = \mathcal{P}(h) x(0)$.  The main challenge is to compute efficiently $\mathcal{P}^{[p]}(h)$, or at least its order-$p$ approximation, given that $I_{p+1}(h) = \mathcal{O}(h^{p+1})$. A notable result, first established in \cite{iserles99ots}, shows that if one employs a quadrature rule of order $p = 2\ell$ with nodes $c_i$, $i = 1,\ldots,\mu$, then $\mathcal{P}^{[2\ell]}(h)$ can be approximated up to order $2\ell$ in $h$ using only the values of $A$ at the quadrature nodes, namely $A_j = A(c_j h)$, $j = 1,\ldots,\mu$. {
	This result has been subsequently proved in a beautiful manner in \cite{zanna98nsi}: this proof is very short, using a clever, indirect argument.} A complete and self-contained proof of this remarkable property is provided in Appendix~\ref{appendixA} 
	{
		highlighting in detail  salient features of the underlying quadrature}.

In line with the treatment carried out there, $\mathcal{P}^{[p=2 \ell]}(h)$ can also be expressed in terms of 
the $\ell$ functions 
\begin{equation} \label{al_qua}
   \alpha_{k+1}^{(2\ell)}(h)=h\sum_{j=1}^{\mu}d_{k,j}A_j, \qquad k = 0, 1, 2, \ldots, \ell -1
\end{equation}
for certain coefficients $d_{k,j}$, that furnish approximations to the derivatives of $A$ at the midpoint $t=h/2$, and satisfy $\alpha_{k+1}^{(2\ell)}(h)={\cal O}(h^{k+1})$. 
The explicit form of the relation \eqref{al_qua} for $\ell = 1,2,3$ is given in equations \eqref{g_or2}--\eqref{g_or6} for Gauss--Legendre quadratures, where $\mu=\ell$.

It turns out that working with these functions is particularly appropriate when designing numerical schemes \cite{blanes09tma}: the resulting
order $p=2 \ell$ method is expressed in terms of $\{\alpha_k^{(2\ell)}\}_{k=1}^{\ell}$, whereas practical implementation is done in terms
of quadratures, by taking into account the representation \eqref{al_qua}. 

Straightforward computation  shows in particular that (cf.\  Appendix~\ref{appendixA})
\begin{equation} \label{expre_Ps}
\begin{aligned}
 &    {\cal P}^{[2]}(h) = I + \alpha_1^{(2)}+ \frac12 \big(\alpha_1^{(2)}\big)^2 \\
 & {\cal P}^{[4]}(h) =  I + \alpha_1^{(4)}+ \frac12 \big(\alpha_1^{(4)}\big)^2 -\frac1{12} \big(\alpha_1^{(4)} \alpha_2^{(4)} -\alpha_2^{(4)} \alpha_1^{(4)} \big)+ \frac16 \big( \alpha_1^{(4)} \big)^3+ \frac1{4!} \big(\alpha_1^{(4)} \big)^4
\\
&  \qquad \qquad -\frac1{12} \Big[
\alpha_1^{(4)} \big(\alpha_1^{(4)} \alpha_2^{(4)} -\alpha_2^{(4)} \alpha_1^{(4)} \big)+ \big(\alpha_1^{(4)} \alpha_2^{(4)} -\alpha_2^{(4)} \alpha_1^{(4)} \big)\alpha_1^{(4)} 
\Big], 
 \end{aligned}
\end{equation} 
and ${\cal P}^{[2\ell]}(h)={\cal P}(h)+{\cal O}(h^{2\ell+1}).$

Although in principle the expressions of ${\cal P}^{[2 \ell]}(h)$ given by \eqref{expre_Ps} can be employed  to construct numerical approximations by expressing the 
functions $\alpha_k^{(2\ell)}$ in terms of the corresponding quadrature, i.e., by using \eqref{al_qua}, the resulting schemes present several disadvantages. In particular, being polynomial approximations, they display instabilities in problems associated with unbounded operators, and thus require a very small  time step $h$. Furthermore, when $A(t)$ evolves in a Lie algebra, the numerical solution does not evolve in the corresponding Lie group.

In this respect, Magnus integrators 
are particularly appealing, since they are formulated as matrix exponentials of elements belonging to the Lie algebra,
\[
   x_{n+1} = \e^{\Omega^{[2 \ell]}(h)} x_n.
\]   
In fact, the expression for $\Omega^{[2\ell]}(h)$ can be obtained directly from ${\cal P}^{[2 \ell]}(h)$ by requiring that
\begin{displaymath}
 \e^{\Omega^{[2 \ell]}(h)}={\cal P}^{[2 \ell]}(h) + {\cal O}(h^{2 \ell+1}), \qquad \mbox{ or } \qquad 
 \Omega^{[2 \ell]}(h)= \ln\left({\cal P}^{[2 \ell]}(h)\right)  + {\cal O}(h^{2\ell+1}),
\end{displaymath}
thus arriving at the following expressions for $\ell=1,2,3$ (see \cite{blanes09tma} for detailed treatment):
\begin{itemize}
	\item Order two: 
	\begin{equation} \label{Mor2} 
	\Omega^{[2]}(h)=\alpha_1^{(2)}.
	\end{equation}
	\item Order four: 
	\begin{equation} \label{Mor4}
	\Omega^{[4]}(h)=\alpha_1^{(4)} + \frac1{12}[\alpha_2^{(4)},\alpha_1^{(4)}].
	\end{equation}
	\item Order six: 
 \begin{equation} \label{Mor6}
   \begin{aligned}
	& \Omega^{[6]}(h) = \alpha_1^{(6)} + \frac1{12}\alpha_3^{(6)}  
	+ \frac1{12}[\alpha_2^{(6)} ,\alpha_1^{(6)} ] + \frac1{240}[\alpha_2^{(6)} ,\alpha_3^{(6)} ] \\
	& \qquad\qquad + \frac1{360}[\alpha_1^{(6)} ,[\alpha_1^{(6)} ,\alpha_3^{(6)} ]] 
	- \frac1{240}[\alpha_2^{(6)} ,[\alpha_1^{(6)} ,\alpha_2^{(6)} ]] \\
	& \qquad\qquad + \frac1{720}[\alpha_1^{(6)} ,[\alpha_1^{(6)} ,[\alpha_1^{(6)} ,\alpha_2^{(6)} ]]]. 
  \end{aligned}
\end{equation}
\end{itemize}
Whereas Magnus integrators involve commutators, other commutator-free approximations are possible. Thus, in particular,
\begin{equation} \label{cf_or4}
  x_{n+1}= \exp \left(\frac12\alpha_1^{(4)} +\frac16\alpha_2^{(4)}  \right) \, \exp \left(\frac12\alpha_1^{(4)} -\frac16\alpha_2^{(4)}  \right) \, x_n
\end{equation}
also constitutes a Lie-group integrator of order 4. Higher-order methods of this type have been presented and analysed in \cite{alvermann11hoc,blanes06fas,blanes17hoc}. {
	As said before, the actual methods are implemented by inserting the relation \eqref{al_qua} into \eqref{Mor2}-\eqref{cf_or4}}.

\setcounter{equation}{0}
\section{Modified Cayley--Magnus integrators}
\label{sect.3}

We are now in a position to construct rational methods of the form \eqref{eq:Cayley_m}. As with Magnus integrators, it is more convenient for that purpose to work
with the functions $\alpha_1^{(2\ell)}(h), \ldots, \alpha_{\ell}^{(2\ell)}(h)$ given in \eqref{al_qua}, so that the schemes read
\begin{equation} \label{prod_m}
\begin{aligned}
\Psi_h^{[2\ell]}= & \prod_{k=1}^m
\left[1-\frac12(w_{k,1}\alpha_1^{(2\ell)}+\cdots +w_{k,\ell}\alpha_{\ell}^{(2\ell)})\right]^{-1} \\
& \qquad \mbox{}\times \left[1+\frac12(w_{k,1}\alpha_1^{(2\ell)}+\cdots +w_{k,\ell}\alpha_{\ell}^{(2\ell)})\right].
\end{aligned}
\end{equation}
In the sequel, and for  simplicity of notation, we  often omit the superscript $(2\ell)$ from $\alpha_j^{(2\ell)}$ when there is no prospect of confusion with the order considered. In any case,
the coefficients $w_{k,j}$ will be determined by requiring that 
\begin{equation} \label{expo_oc}
 \Psi_h^{[2 \ell]} = {\cal P}^{[2 \ell]}(h)+{\cal O}(h^{2\ell+1}) = \e^{\Omega^{[2\ell]}(h)} + \mathcal{O}(h^{2 \ell +1}).
\end{equation}
In this respect, the analysis simplifies by taking into account the existing relationship between the exponential and Cayley maps, namely
\[
\Cay (x) = \frac{1+x/2}{1-x/2} = \e^{2\arctan(x/2)}=
\exp\!\left(x+\frac1{12}x^3+\frac1{80}x^5+\cdots\right), 
\]
for $x\in(-2,2)$. Then, clearly, for $k=1,\ldots, \ell$, we have
\begin{eqnarray} \label{cay_exp}
	\Cay \Big(\sum_{j=1}^{\ell} w_{k,j} \, \alpha_j(h) \Big) &\!\!\!=\!\!\!& \left(I-\frac12\sum_{j=1}^{\ell} w_{k,j} \, \alpha_j(h)\right)^{\!-1}\! \left(I+\frac12\sum_{j=1}^{\ell} w_{k,j}\, \alpha_j(h)\right) \\
	& \!\!\!=\!\!\! & \exp\!\left(\sum_{j=1}^{\ell} w_{k,j} \, \alpha_j(h) + \frac1{12}\Big(\sum_{j=1}^{\ell} w_{k,j} \, \alpha_j(h)\Big)^3+ {\cal O}(h^5)  \right)\!.  \nonumber
\end{eqnarray}
The main idea is to substitute \eqref{cay_exp} into the product \eqref{prod_m} defining the method $\Psi_h^{[2\ell]}$, subsequently combining the resulting product of exponentials into a single exponential using the Baker--Campbell--Hausdorff (BCH) formula. Finally, this expression is compared with $\Omega^{[2\ell]}(h)$ in line with \eqref{expo_oc}. This procedure yields a set of order conditions that must be satisfied by the coefficients $w_{k,j}$.

Since both the exact solution and the truncated Magnus expansion are time-symmetric, the following result provides a practical guideline to ensure that the method \eqref{prod_m} also preserves this property: the composition must be symmetric with respect to the elements with odd indices, $\alpha_{2j+1}(h)$, and skew-symmetric with respect to those with even indices, $\alpha_{2j}(h)$. This follows from the identities $\alpha_{2k-1}(-h) = -\alpha_{2k-1}(h)$ and $\alpha_{2k}(-h) = \alpha_{2k}(h)$, for $k = 1,2,\ldots$ (cf.\  Appendix~\ref{appendixA}).

 \begin{proposition} \label{thm:Cayle_symmetry}
	Let $\Psi(h)$ be a time-symmetric map, i.e.,  one that satisfies $\Psi(h)^{-1}=\Psi(-h)$. Then, the product
	\[
	\Phi(h)=\Cay \Big(\sum_{j=1}^{\ell} z_j \, \alpha_j(h)\Big) \cdot \Psi(h) \cdot \Cay \Big(\sum_{j=1}^{\ell} (-1)^{j+1}z_j \, \alpha_j(h)\Big),
	\]
	with $z_j \in \mathbb{R}$, is also time-symmetric: $\Phi(h)^{-1}=\Phi(-h)$.
\end{proposition}
\begin{proof}
	It follows from the definition of the Cayley map and the symmetry properties of the functions $\alpha_k(h)$ that
	\[
	\Cay^{-1} \Big(\sum_{j=1}^{\ell} z_j \, \alpha_j(h)\Big) =
	\Cay \Big(-\sum_{j=1}^{\ell} z_j \, \alpha_j(h)\Big) =
	\Cay \Big(\sum_{j=1}^{\ell} (-1)^{j+1}z_j \, \alpha_j(-h)\Big).
	\]
Therefore
\begin{eqnarray*}
  \Phi(h)^{-1} &\!\!\!=\!\!\!& 
  \Cay^{-1} \Big(\sum_{j=1}^{\ell}(-1)^{j+1} z_j \, \alpha_j(h)\Big) \cdot \Psi(h)^{-1} \cdot  \Cay^{-1} \Big(\sum_{j=1}^{\ell} z_j \, \alpha_j(h)\Big)
  \\ &\!\!\!=\!\!\!& \Cay \Big(\sum_{j=1}^{\ell} z_j \, \alpha_j(-h)\Big) \cdot \Psi(-h) \cdot \Cay \Big(\sum_{j=1}^{\ell} (-1)^{j+1}z_j \, \alpha_j(-h)\Big)\\ 
  &\!\!\!=\!\!\!& \Phi(-h).
\end{eqnarray*}
\end{proof}

The procedure to construct methods of this class is best illustrated by considering the simplest 4th-order case.

\subsection{Time-symmetric 4th-order modified Cayley--Magnus integrators}

Only the two functions $\alpha_1^{(4)}(h)$ and $\alpha_2^{(4)}(h)$ are necessary to achieve order 4, whereby the Cayley transform \eqref{cay_exp}
reads explicitly
\begin{equation} \label{cayley_4o}
\begin{aligned}
 &  \Cay \big(w_{k,1}\alpha_1+ w_{k,2} \alpha_2 \big) = \Big[I-\frac12(w_{k,1} \alpha_1+ w_{k,2} \alpha_2)\Big]^{-1}  \Big[I+\frac12( w_{k,1} \alpha_1+ w_{k,2}\alpha_2)\Big] \\
  & \qquad\quad =  \exp\left(( w_{k,1} \alpha_1+ w_{k,2} \alpha_2) + \frac1{12}( w_{k,1} \alpha_1+ w_{k,2} \alpha_2)^3+ {\cal O}(h^5)  \right).
\end{aligned}
\end{equation}
Consistently with Proposition~\ref{thm:Cayle_symmetry}, the composition defining a time-symmetric 
method is of the form  
\begin{equation} \label{compo_cayley_4}
\begin{aligned}
 & \Psi_h^{[4]} =\Cay (w_{k,1} \, \alpha_1+ w_{k,2} \, \alpha_2)\cdots
\Cay (w_{1,1} \, \alpha_1+w_{1,2} \, \alpha_2) \\
& \qquad\quad \Cay (w_{1,1} \, \alpha_1-w_{1,2} \, \alpha_2)\cdots
\Cay (w_{k,1} \, \alpha_1-w_{k,2} \, \alpha_2)
\end{aligned}
\end{equation}
if $m=2k$ (even). We  adopt the short-hand notation
\[
 \Psi_h^{[4]}:  \Big( (w_{k,1},w_{k,2}) \, \cdots \, 
(w_{1,1},w_{1,2}) \, (w_{1,1},-w_{1,2}) \, \cdots  \, (w_{k,1},-w_{k,2})\Big).
\]
If, on the other hand, $m=2k-1$ (odd), then the following composition involves $2k-1$ parameters:
\[
  \Psi_h^{[4]}:  \Big( (w_{k,1},w_{k,2}) \, \cdots \, 
(w_{2,1},w_{2,2}) \, (w_{1,1},0) \, (w_{2,1},-w_{2,2}) \, \cdots \, (w_{k,1},-w_{k,2})\Big).
\]
Since the most general Lie algebra generated by symmetric Cayley maps up to order three contains only the elements
${\alpha_1, \alpha_1^3, [\alpha_2,\alpha_1]}$ \cite{iserles00otd}, the explicit construction of $\Psi_h^{[4]}$ in either case requires handling products of the form
\[
\e^{C_4^{[n+1]}} = 
\Cay (w_{n+1,1}\alpha_1+w_{n+1,2}\alpha_2)\cdot \e^{C_4^{[n]}} \cdot  \Cay (w_{n+1,1}\alpha_1-w_{n+1,1}\alpha_2), 
\]
where
\[
C_4^{[n]} = 
\gamma_1^{[n]} \alpha_1 +
\gamma_{3,1}^{[n]} \, [\alpha_2,\alpha_1] +
\gamma_{3,2}^{[n]} \, \alpha_1^3 + {\cal O}(h^5), \qquad n=0,1,2,\ldots
\]
Taking into account \eqref{cayley_4o}, successive application of the BCH formula leads to the recursive relation
\begin{eqnarray*}
  \gamma_1^{[n+1]} & \!\!\!=\!\!\! & \gamma_1^{[n]} +2w_{n+1,1} \\
  \gamma_{3,1}^{[n+1]} &\!\!\!=\!\!\! & \gamma_{3,1}^{[n]} +\big(w_{n+1,1} + \gamma_1^{[n]} \big)w_{n+1,2} \\
  \gamma_{3,2}^{[n+1]} & \!\!\!=\!\!\! & \gamma_{3,2}^{[n]} +\frac16 w_{n+1,1}^3,
\end{eqnarray*}
which has to be initialised, for a composition with $2k$ maps, with
\[
  \gamma_1^{[0]} =
\gamma_{3,1}^{[0]}=
\gamma_{3,2}^{[0]}=0 
\]
and, for a composition with $2k-1$ maps, with
\[
\gamma_1^{[0]} = w_{1,1}, \qquad
\gamma_{3,1}^{[0]}=0, \qquad
\gamma_{3,2}^{[0]}= \frac1{12} w_{1,1}^3 .
\]
Now, according with \eqref{expo_oc}, the composition \eqref{compo_cayley_4} must reproduce the truncated Magnus expansion \eqref{Mor4} up to order 3
(due to time-symmetry), and this can be furnished by a composition involving only three Cayley maps,
\begin{displaymath} 
\Psi^{[4]}_h = \Big( (w_{2,1},w_{2,2}) \, (w_{1,1}, 0) \, (w_{2,1},-w_{2,2}) \Big)
\end{displaymath}
as long as the parameters $w_{k,j}$ satisfy the equations
\begin{eqnarray*} 
	&  & 2 w_{2,1}+ w_{1,1}=1  \\
	& & (w_{1,1} + w_{2,1}) w_{2,2}=\frac1{12} \\
	 &  & 2 w_{2,1}^3 +w_{1,1}^3=0, 
\end{eqnarray*}
whose unique real solution is given by
\begin{equation} \label{c_order_43a}
 \displaystyle  
 w_{2,1}=\frac1{2-2^{1/3}}, \qquad w_{1,1}=1-2 w_{2,1}, \qquad
 w_{2,2}=\frac1{12(1-w_{2,1})}.
\end{equation}
As a matter of fact, this corresponds to the method proposed in \cite{maslovskaya26cfc}, and is denoted as $\mathrm{Cay}_34$ in Table \ref{tab:cayley_magnus_methods}. 

Several improvements can be considered, however. In particular, we can incorporate more Cayley maps into the composition, and therefore more parameters,
to try to optimise the leading error term or we can apply higher order quadrature rules.

Consider, for instance, a time-symmetric composition with five Cayley maps instead of the three previously analysed\footnote{A composition of four Cayley maps leads to a system of order conditions with no real solutions.},
\begin{displaymath}
\tilde{\Psi}^{[4]}_h = 
\Big( (w_{3,1}, w_{3,2}) \, (w_{2,1},w_{2,2}) \, (w_{1,1},0) \, (w_{2,1},-w_{2,2}) \, (w_{3,1},-w_{3,2})\Big).
\end{displaymath}
In that case, the order conditions to achieve order four read
\begin{eqnarray*} 
	&  & 2 w_{3,1}+2 w_{2,1}+ w_{1,1}=1 \\
	& & (w_{1,1} + w_{2,1}) w_{2,2} + (w_{1,1} + 2 w_{2,1}+ w_{3,1}) w_{3,2}=\frac1{12} \\
	&  & 2 w_{3,1}^3 +2 w_{2,1}^3 +w_{1,1}^3=0,
\end{eqnarray*}
with two free parameters, allowing for different possibilities. Thus, we
can take, for example, $w_{2,1}=w_{3,1}$ and try to minimise some of the terms at order five. 
The following solution has displayed good performance:
\[
\displaystyle  
w_{3,1}=w_{2,1}=\frac1{4-4^{1/3}}, \qquad w_{1,1}=1-4 w_{2,1},
\]
\[
\displaystyle  
w_{3,2}=\frac7{240(1-2 w_{2,1})}, \qquad 
w_{2,2}=\frac{1-12(1-w_{2,1})w_{3,2}}{12(1-3 w_{2,1})}.
\]
The resulting method is denoted  by 
$\mathrm{Cay}_54$ in Table \ref{tab:cayley_magnus_methods}.

\begin{table}[!htbp]
\centering
\small
\setlength{\tabcolsep}{4pt}
\renewcommand{\arraystretch}{1.2}

\begin{tabularx}{\linewidth}{@{}>{\raggedright\arraybackslash}p{0.16\linewidth}X@{}}
\toprule
\textbf{Method} & \textbf{Composition and coefficients} \\
\midrule

\(\mathrm{Cay}_{3}4\):
&
\(\Bigl((w_{2,1},w_{2,2})\,(w_{1,1},0)\,(w_{2,1},-w_{2,2})\Bigr)\)
\\[-7pt]
&
\[
\begin{aligned}
w_{2,1} &= \frac{1}{2-2^{1/3}}, & 
w_{1,1} &= 1-2w_{2,1}, & 
w_{2,2} &= \frac{1}{12(1-w_{2,1})}.
\end{aligned}
\]
\\[-10pt]

\(\mathrm{Cay}_{5}4\):
&
\(\Bigl((w_{3,1},w_{3,2})\,(w_{2,1},w_{2,2})\,(w_{1,1},0)\,
(w_{2,1},-w_{2,2})\,(w_{3,1},-w_{3,2})\Bigr)\)
\\[-12pt]
&
\[
\begin{aligned}
w_{3,1}=w_{2,1} &= \frac{1}{4-4^{1/3}}, &
w_{1,1} &= 1-4w_{2,1}, \\
w_{3,2} &=\frac7{240(1-2 w_{2,1})}, &
w_{2,2}&=\frac{1-12(1-w_{2,1})w_{3,2}}{12(1-3 w_{2,1})}.
\end{aligned}
\]
\\[-10pt]

\(\mathrm{Cay}_{7}4\):
&
\(\Bigl((w_{4,1},w_{4,2},w_{4,3})\,\cdots\,
(w_{1,1},0,w_{1,3})\,\cdots\,
(w_{4,1},-w_{4,2},w_{4,3})\Bigr)\)
\\[-14pt]
&
\[
\begin{aligned}
w_{1,1} &=  0.9436189826258903,  &
w_{2,1} &= -0.8341605550808652, \\
w_{3,1} &=  0.43117553188396,    &
w_{4,1} &=  w_{3,1}, \\
w_{1,2} &=  0,                   &
w_{1,3} &=  0.884982196784669, \\
w_{2,2} &=  0.06389979531412822, &
w_{2,3} &= -0.6265465634394808, \\
w_{3,2} &=  0.08835088703663657, &
w_{3,3} &=  0.1707144543780912, \\
w_{4,2} &=  0.17979588264059018, &
w_{4,3} &=  0.055007677335721684.
\end{aligned}
\]
\\[-10pt]

\(\mathrm{Cay}_{13}6\):
&
\(\Bigl((w_{7,1},w_{7,2},w_{7,3})\,\cdots\,
(w_{1,1},0,w_{1,3})\,\cdots\,
(w_{7,1},-w_{7,2},w_{7,3})\Bigr)\)
\\[-16pt]
&
\[
\begin{aligned}
w_{1,1} &= -0.6274523445492189,    &
w_{2,1} &=  0.5850565174736707, \\
w_{3,1} &= -0.45967745375388464,   &
w_{4,1} &=  0.172086777138706, \\
w_{5,1} &=  w_{4,1},               &
w_{6,1} &=  w_{4,1}, \\
w_{7,1} &=  w_{4,1},               & \\[-0.4em]
w_{1,2} &=  0,                     &
w_{1,3} &=  0.004329477802178489, \\
w_{2,2} &= -0.0063913535826220485, &
w_{2,3} &= -0.04429205088886197, \\
w_{3,2} &= -0.07233744752005296,   &
w_{3,3} &=  0.06509491660750541, \\
w_{4,2} &= -0.082715747715483,     &
w_{4,3} &= -0.03516880921224163, \\
w_{5,2} &=  0.0052328434008880416, &
w_{5,3} &=  \tfrac{1}{35}, \\
w_{6,2} &=  0.0049981606172231335, &
w_{6,3} &= -\tfrac{1}{55}, \\
w_{7,2} &=  \tfrac{1}{12},          &
w_{7,3} &=  \tfrac{1}{23}.
\end{aligned}
\]
\\[-4pt]

\bottomrule
\end{tabularx}
\caption{Modified Cayley--Magnus methods $\mathrm{Cay}_sp$ of orders four and six considered in the present work. The subindex $s$ refers to the number of
Cayley transforms involved, whereas $p$ denotes the order.}
\label{tab:cayley_magnus_methods}

\end{table}


\subsection{Time-symmetric 6th-order modified Cayley--Magnus integrators}

In principle, a similar procedure can be extended to higher orders; however, significant technical difficulties arise due to the rapidly increasing number of order conditions. It is therefore preferable to adopt a systematic approach that accounts for the dimensional structure, as well as the various combinations of the functions $\alpha_j$ involved,  consequently determining the number of Cayley maps required to achieve the desired order. This can be accomplished by observing that the quadratic Lie algebra $o_N(\mathbb{C},J)$ defined in \eqref{alg1} is, in fact, a hierarchical Lie algebra \cite{iserles00otd}.

Let us consider the generators $X_1,X_2,\ldots,X_m$ of a hierarchical algebra $\mathfrak{g}$, with  $X_i$ of grade $i$, and the  product $[[\cdots]]_m$ 
defined as
\[
[[X_{i_1},X_{i_2},\ldots,X_{i_m}]]_m=
X_{i_1},X_{i_2},\ldots,X_{i_m}-(-1)^{m}X_{i_m},\ldots,X_{i_2},X_{i_1}.
\]
This product obeys the alternating symmetry, multilinearity and hierarchy conditions  \cite{iserles00otd}. Thus, $\mathfrak{g}$ contains the terms $X_i$ and all possible combinations by means of each $m$-nary operation. This is a hierarchical algebra which is unique up to an isomorphism. 

In \cite{iserles00otd} it is shown how to compute the dimension of the algebra at each grade and how to obtain a basis. Thus, let us assume that we have the
generators $X_1,X_2,X_3$, with each $X_i$ of grade $i$. Then, a basis of the graded algebra up to order five is given by 
\[
\begin{aligned}
&	{\cal B}^3_1  =  \{ X_1\};   \\
&	{\cal B}^3_2  =  \{ X_2\};   \\
&	{\cal B}^3_3  =  \{ X_3,[[X_1,X_2]]_2,[[X_1,X_1,X_1]]_3\};   \\
&	{\cal B}^3_4  =  \{ [[X_1,X_3]]_2,[[X_1,X_1,X_2]]_3,[[X_1,X_2,X_1]]_3\};   \\
&	{\cal B}^3_5  =  \{ [[X_2,X_3]]_2,[[X_1,X_2,X_2]]_3,[[X_1,X_1,X_3]]_3,[[X_2,X_1,X_2]]_3,[[X_1,X_3,X_1]]_3,\!\!  \\
	&  \qquad\quad  [[X_1,X_1,X_1,X_2]]_4, [[X_1,X_1,X_2,X_1]]_4,[[X_1,X_1,X_1,X_1,X_1]]_5\}.
\end{aligned}
\]
Note that, since $\alpha_i^{(6)}(h) = \mathcal{O}(h^i)$, we may take $X_i=\alpha_i^{(6)}(h)$, $i=1,2,3$, identify an appropriate basis and therefore form 
an independent set of order conditions for a method of order six. Other bases are more convenient for our analysis, however. Thus, in particular, we choose
\begin{eqnarray*}
	{\cal B}^3_1 & = & \{ E_{1,1}=\alpha_1\};   \\
	{\cal B}^3_2 & = & \{ E_{2,1}=\alpha_2\};   \\
	{\cal B}^3_3 & = & \{ E_{3,1}=\alpha_3, \ E_{3,2}=[\alpha_2,\alpha_1], \ E_{3,3}=\alpha_1^3\};   \\
	{\cal B}^3_4 & = & \{ E_{4,1}=[\alpha_1,\alpha_3], \ E_{4,2}=[\alpha_1,[\alpha_1,\alpha_2]], \
	E_{4,3}=(\alpha_1^2\alpha_2+\alpha_1\alpha_2\alpha_1+\alpha_2\alpha_1^2)\};   \\
	{\cal B}^3_5 & = & \{ E_{5,1}=[\alpha_2,\alpha_3]], \
	E_{5,2}=[\alpha_1,[\alpha_1,\alpha_3]],  \
	E_{5,3}=[\alpha_2,[\alpha_1,\alpha_2]], \\
	&  &  
\ \, E_{5,4}=[\alpha_1,[\alpha_1,[\alpha_1,\alpha_2]]], \ 
E_{5,5}=\alpha_1^5,\ E_{5,6}=(\alpha_1\alpha_2^2+\alpha_2\alpha_1\alpha_2+\alpha_2^2\alpha_1), \	\\
	&  &   \ \, E_{5,7}=(\alpha_1^2\alpha_3+\alpha_1\alpha_3\alpha_1+\alpha_3\alpha_1^2), \ E_{5,8}=[E_{1,1},E_{4,3}] = [\alpha_1^3,\alpha_2]\}. \
\end{eqnarray*}

Equipped with this information, we can proceed now to derive 6th-order time-symmetric schemes. We have to deal with products of the form
\begin{equation} \label{cay_zass}
\e^{C_6^{[n+1]}} = 
\Cay (x_n\alpha_1+y_n\alpha_2+z_n\alpha_3)\, \e^{C_6^{[n]}}  \Cay (x_n\alpha_1-y_n\alpha_2+z_n\alpha_3)
\end{equation}
to construct a time-symmetric composition, with
\[
 C_6^{[n]} = \gamma_{1,1}^{[n]} E_{1,1} + \sum_{j=1}^3 \gamma_{3,j}^{[n]} E_{3,j} + \sum_{j=1}^8 \gamma_{5,j}^{[n]} E_{8,j}.
\]
Note that $C_6^{[n]}$ can  be written as $C_6^{[n]} = C_{6,1}^{[n]} + C_{6,2}^{[n]}$, where
\[
 \begin{aligned}
   C_{6,1}^{[n]} & = \gamma_{1,1}^{[n]} E_{1,1} + \gamma_{3,1}^{[n]} E_{3,1} 
	+ \gamma_{3,2}^{[n]} E_{3,2}+ \gamma_{5,1}^{[n]} E_{5,1} + \gamma_{5,2}^{[n]} E_{5,2} 
	+ \gamma_{5,3}^{[n]} E_{5,3} + \gamma_{5,4}^{[n]} E_{5,4}, \\
  C_{6,2}^{[n]} & = \gamma_{3,3}^{[n]} E_{3,3} + \gamma_{5,5}^{[n]} E_{5,5} + 
	\gamma_{5,6}^{[n]} E_{5,6} + \gamma_{5,7}^{[n]} E_{5,7} + \gamma_{5,8}^{[n]} E_{5,8}.
\end{aligned}
\]
This decomposition is particularly convenient for deriving the order conditions, since $C_{6,1}^{[n]}$ contains exactly the terms appearing in the Magnus expansion up to order 6 (see \eqref{Mor6}). Consequently, $C_{6,2}^{[n]}$ can be interpreted as the additional contribution arising from the use of Cayley transforms instead of exponentials of the form $\exp\big(\sum_{j=1}^3 w_{k,j} \alpha_j\big)$ in the composition, as is done in commutator-free exponential integrators \cite{blanes06fas}.

Computing explicitly the product \eqref{cay_zass} up to order 6, we obtain the values
\allowdisplaybreaks
\begin{align*}
&	\gamma_{1,1}^{[n+1]}  =  \gamma_{1,1}^{[n]} +2x_n, \\
&	\gamma_{3,1}^{[n+1]}  =  \gamma_{3,1}^{[n]} +2z_n, \\
&	\gamma_{3,2}^{[n+1]}  =  \gamma_{3,2}^{[n]} -\frac12(\gamma_{1,1}^{[n+1]} + \gamma_{1,1}^{[n]}) \, y_n, \\
&	\gamma_{3,3}^{[n+1]}  =  \gamma_{3,3}^{[n]} +\frac16x_n^3, \\
&	\gamma_{5,1}^{[n+1]}  =  \gamma_{5,1}^{[n]} +\frac12(\gamma_{3,1}^{[n+1]} + \gamma_{3,1}^{[n]}) \, y_n,  \\
&	\gamma_{5,2}^{[n+1]}  =  \gamma_{5,2}^{[n]} + \frac1{12}(x_n-\gamma_{1,1}^{[n]})(\gamma_{3,1}^{[n]}x_n-\gamma_{1,1}^{[n]}z_n) \\
& \qquad -
	\frac1{12}(x_n+\gamma_{1,1}^{[n+1]})(\gamma_{3,1}^{[n+1]}x_n-\gamma_{1,1}^{[n+1]}z_n), \\
&	\gamma_{5,3}^{[n+1]}  =  \gamma_{5,3}^{[n]} +\frac12(\gamma_{3,2}^{[n+1]} + \gamma_{3,2}^{[n]})y_n+\frac1{12}(\gamma_{1,1}^{[n+1]} - \gamma_1^{[n]})y_n^2 ,\\
&	\gamma_{5,4}^{[n+1]}  =  \gamma_{5,4}^{[n]}  +\frac1{12}((x_n-\gamma_{1,1}^{[n]})\gamma_{3,2}^{[n]}x_n-(\gamma_{1,1}^{[n+1]}+x_n)\gamma_{3,2}^{[n]}x_n) \\
& \qquad +
	\frac1{24}((\gamma_1^{[n]})^2-\gamma_1^{[n+1]})^2))x_ny_n,\\
&	\gamma_{5,5}^{[n+1]}  =  \gamma_{5,5}^{[n]}  +\frac1{40}x_n^5, \\
&	\gamma_{5,6}^{[n+1]}  =  \gamma_{5,6}^{[n]} + \frac16x_ny_n^2,\\
&	\gamma_{5,7}^{[n+1]}  =  \gamma_{5,7}^{[n]} + \frac16x_n^2z_n,\\
&	\gamma_{5,8}^{[n+1]}  =  \gamma_{5,8}^{[n]} -\frac12 (\gamma_{3,3}^{[n+1]} + \gamma_{3,3}^{[n]}) y_n-\frac1{24} (\gamma_{1,1}^{[n+1]} + \gamma_{1,1}^{[n]}) x_n^2y_n,
\end{align*}
to be initiated, for a composition with $2k$ Cayley maps, with 
\[
\gamma_{i,j}^{[0]}=0, \qquad \mbox{ for all }  \ i,j.
\]

Once we are dealing instead with a composition with $2k-1$ Cayley maps,  the recurrence is initialized for $n=1$ by
\begin{eqnarray*}
	& & 
	\gamma_1^{[1]} = x_{1}, \qquad
	\gamma_{3,1}^{[1]}= z_{1}, \qquad
	\gamma_{3,2}^{[1]}=0, \qquad
	\gamma_{3,3}^{[1]}= \frac1{12} x_{1}^3,  \\
	& & 
	\gamma_{5,1}^{[1]} =0, \qquad
	\gamma_{5,2}^{[1]}= 0, \qquad
	\gamma_{5,3}^{[1]}=0, \qquad
	\gamma_{5,4}^{[1]}= 0,  \\
	& & 
	\gamma_{5,5}^{[1]} = \frac1{80} x_{1}^5, \qquad
	\gamma_{5,6}^{[1]}= 0, \qquad
	\gamma_{5,7}^{[1]}= \frac1{12} x_{1}^2z_{1}, \qquad
	\gamma_{5,8}^{[1]}=0. 
\end{eqnarray*}

In either case, the order conditions are obtained by taking $C_6^{[k]} = \Omega^{[6]}(h)$ as given by \eqref{Mor6}, or explicitly
\[
\begin{aligned}
& \gamma_{1,1}^{[k]} = 1, \
\gamma_{3,1}^{[k]} = \frac1{12}, \
\gamma_{3,2}^{[k]} = \frac1{12}, \
\gamma_{5,1}^{[k]}  = \frac1{240}, \
\gamma_{5,2}^{[k]} = \frac1{360}, \
\gamma_{5,3}^{[k]} = -\frac1{240},  \\
& \gamma_{5,4}^{[k]} = \frac1{720}, \; \gamma_{3,3}^{[k]} = 0, \; \gamma_{5,5}^{[k]} = 0, \; \gamma_{5,6}^{[k]} = 0, \; \gamma_{5,7}^{[k]} = 0, \; \gamma_{5,8}^{[k]} = 0,
\end{aligned}
\]
for a total of 12 polynomial equations in the coefficients of the composition. In consequence, at least eight Cayley maps are required to obtain a number of free parameters equal to the number of order conditions. This should be compared with the minimum of five exponentials acting on vectors needed in standard commutator-free methods of order six. {
	However, the conditions associated to $\gamma_{1,1}^{[4]},\gamma_{3,3}^{[4]},\gamma_{5,5}^{[4]}$
	\[
	\gamma_{1,1}^{[4]}=2\sum_{j=1}^4x_j=1, \qquad \gamma_{3,3}^{[4]}=\frac16\sum_{j=1}^4x_j^3=0, \qquad \gamma_{5,5}^{[4]}=\frac1{40}\sum_{j=1}^4x_j^5=0
	\]
	admit no real solutions, and therefore additional maps must be introduced.}

In practice, we have observed that, for problems evolving smoothly in time, the dominant error contribution arises from the term $\alpha_1^{7}$, which is multiplied by the factor 
$2\sum_{j=1}^k w_{j,1}^7$ or $w_{1,1}^7 + 2\sum_{j=2}^k w_{j,1}^7$, 
depending on whether the composition has an even or odd number of stages, respectively.

For this reason, it has proven convenient to optimise the coefficients $w_{j,1}$ so that they satisfy the order conditions associated with $\gamma_{1,1}^{[n]}$, $\gamma_{3,3}^{[n]}$, and $\gamma_{5,5}^{[n]}$, while simultaneously minimising or eliminating the coefficients of $\alpha_1^{7}$ and, possibly, $\alpha_1^{9}$. Once the coefficients $w_{j,1}$ are fixed, the remaining coefficients $w_{j,2}$ and $w_{j,3}$ must satisfy the nine remaining equations. Notably, five of these equations depend solely on the coefficients $w_{j,2}$.

We have analysed compositions involving up to 13 Cayley maps, and the most effective scheme we have identified consists precisely of 13 maps,
\begin{displaymath}
	\Big((w_{7,1}, w_{7,2}, w_{7,3}) \, \cdots \,  (w_{1,1}, 0, w_{1,3}) \,  \cdots \, 
	(w_{7,1}, -w_{7,2}, w_{7,3}) \Big),
\end{displaymath}
when the coefficients $w_{j,1}$ are chosen such that $w_{1,1}^7 + 2\sum_{j=2}^k w_{j,1}^7 = 0$
and minimise the quantity $w_{1,1}^9 + 2\sum_{j=2}^7 w_{j,1}^9$. The coefficients of this method, denoted by $\mathrm{Cay}_{13}6$ are collected in
Table \ref{tab:cayley_magnus_methods}.

Alternatively, by incorporating more Cayley maps than strictly required and/or using higher order quadrature rules, the additional free parameters can be exploited to construct optimised fourth-order integrators. Thus, the method denoted by $\mathrm{Cay}_74$ in Table~\ref{tab:cayley_magnus_methods}, which involves only seven Cayley maps, is of order
 four, while using a sixth-order quadrature rule so that  all order conditions up to order 5 are satisfied, except for those associated with $E_{5,3}$ and $E_{5,6}$.

\setcounter{equation}{0}
\section{Numerical example: A high-dimensional Rosen--Zener model}
\label{sect.4}

To illustrate the  schemes presented in this paper in practice, we consider next 
a  high-dimensional generalisation of the Rosen--Zenner model for a two-level quantum system  \cite{rosen32dsg}, which is incidentally
closely related to the system studied in \cite{bonhomme23anf,kyoseva07pro}.
The corresponding Schr{\"o}dinger equation  for the evolution operator in the interaction picture is of the form
\begin{equation} \label{r-z.1}
		\begin{cases}
			U'(t) = 
			- \, i \, H(t) \, U(t)\,, \qquad t \in (t_0, T)\,, \\
			U(t_0) = I,\,
		\end{cases}
\end{equation}
	where the time-dependent Hamiltonian reads, after normalisation,
\begin{equation} \label{r-z.2}
		\begin{gathered}
			H(t) = f_1(t) \, \sigma_3 \otimes I  + f_2(t) \, \sigma_1 \otimes R  \in \mathbb{C}^{d \times d}\,, \qquad d = 2 \, k\,. \\
		\end{gathered}
\end{equation}
Here $I \in \R^{k \times k}$ is the identity matrix, 
\[
	\sigma_1 = \begin{pmatrix} 0 & 1 \\ 1 & 0 \end{pmatrix}\,, \qquad 
	\sigma_2 = \begin{pmatrix} 0 &  - \, i \\ i  & 0 \end{pmatrix}\,, \qquad 
	\sigma_3 = \begin{pmatrix} 1 &  0 \\ 0  & -1 \end{pmatrix}\,,
\]
are  Pauli matrices, and $R = \mbox{tridiag}\big(1, 0, 1\big) \in \R^{k \times k}$. Alternatively,  we can write the Hamiltonian in the form
\begin{displaymath}
	H(t) = \begin{pmatrix} f_1(t) I & f_2(t) R \\ f_2(t) R &  -f_1(t) I \end{pmatrix}.
\end{displaymath}

Before selecting specific parameter values and performing numerical tests, it is reasonable to outline some preliminary considerations with respect to the
application of the previous modified Cayley--Magnus integrators for solving this problem.
These methods require to solve at each stage a linear system with the following structure,
\begin{equation}\label{eq:10}
	(I-hH_j)u = \begin{pmatrix} (1-a_j) I & b_j R \\ b_j R & (1+a_j) I \end{pmatrix} 
	\begin{pmatrix} u_1 \\ u_2  \end{pmatrix} = 
	\begin{pmatrix} d_1  \\ d_2  \end{pmatrix}, 
\end{equation}
where 
\[
a_j=h\sum_{i=1}^{\ell} x_{j,i}f_1(t_n+c_ih), \qquad\quad
b_j=-h\sum_{i=1}^{\ell} x_{j,i}f_2(t_n+c_ih), 
\]
and appropriate coefficients $x_{j,i}$ from each method. To solve \eqref{eq:10} we can proceed as follows: first we solve the system
\[
((1- a_j^2)I- b_j^2R^2)u_2=v, \qquad v=(1-a_j)d_2 - b_j R d_1.
\]


The vector $v$ can be computed using $2k$ multiplications and $2k$ additions (i.e., $4k$ flops). Here, $1 - a_k$ can be precomputed, and the action of the matrix $R$ on a vector involves only $k$ additions. The associated linear system can then be solved using a variant of the Thomas algorithm for tridiagonal systems with a cost of $6k$ flops.

Note that $R^2$ is pentadiagonal, although only three of its diagonals are nonzero. Finally, the quantity
\[
u_1 = \frac1{1-a_j}(d_1 - b_j R u_2)
\]
is computed with $4k$ flops, resulting in a total cost of $14k$ flops.
Moreover, taking into account that
\begin{equation}\label{eq:Cay_12}
\Cay(A)u = (I - A/2)^{-1}(I + A/2)u = 2(I - A/2)^{-1}u - u,
\end{equation}
the cost of applying a Cayley map amounts to adding $k$ multiplications and $k$ additions to the cost of solving the linear system, leading to a total of $16k$ flops.
For comparison, a single matrix--vector product $H_j u$ requires $4k$ multiplications and $4k$ additions, i.e., $8k$ flops. Thus, a Cayley map has a computational cost comparable to two matrix--vector products, while additionally preserving unitarity (note that computing the full matrix $\Cay(A)$ bears the cost of ${\cal O}(k^2)$ flops).

For our numerical experiments, we take
\[
		f_1(t) = V_0 \frac{\cos(\omega \, t)}{\cosh\big(\tfrac{t}{T_0}\big)}, \qquad\qquad 
		f_2(t) = - \, V_0 \frac{\sin(\omega \, t)}{\cosh\big(\tfrac{t}{T_0}\big)}
\]
with $k = 50$, $T_0 = 1$. In addition, to assess the influence of $V_0$ (and thus the norm of the matrix) and $\omega$ (which governs the contribution of time derivatives) on the relative performance of the methods, we consider the following parameter values:
\[
 \mbox{ (a) } \, V_0 = 10, \ \omega = 5; \qquad \mbox{ (b) } \, V_0 = 10, \ \omega = 10; \qquad \mbox{ (c) } \, V_0 = 20, \ \omega = 5.
\]
We integrate equation~\eqref{r-z.1} from $t_0 = -4T_0$ up to the final time $t_f = 4T_0$, and compute numerical approximations $U_{\text{app}}(t_f,t_0)$ at $t = t_f$ for various time step sizes. A reference solution, $U_{\text{ref}}(t_f,t_0)$, is obtained numerically with high accuracy. In practice, however, rather than computing the full fundamental matrix, one typically evaluates its action on a vector (corresponding to the initial condition).

We first compare the relative performance of the modified Cayley--Magnus methods listed in Table~\ref{tab:cayley_magnus_methods}, together with the second-order Cayley method \eqref{eq:Cay_12}, denoted by Cay$_12$. The corresponding results are displayed in Figure~\ref{fig00}.
The superiority of the new methods proposed in this work is clearly evident. It is also worth noting that the second-order Cayley method Cay$_12$
is competitive only at relatively low accuracy requirements and when the norm of the matrix is moderate. As the norm increases, or higher accuracies are sought, the new modified Cayley--Magnus methods consistently outperform the second-order scheme.

\begin{figure}[h!]
	\begin{center}
		\includegraphics[scale=0.54]{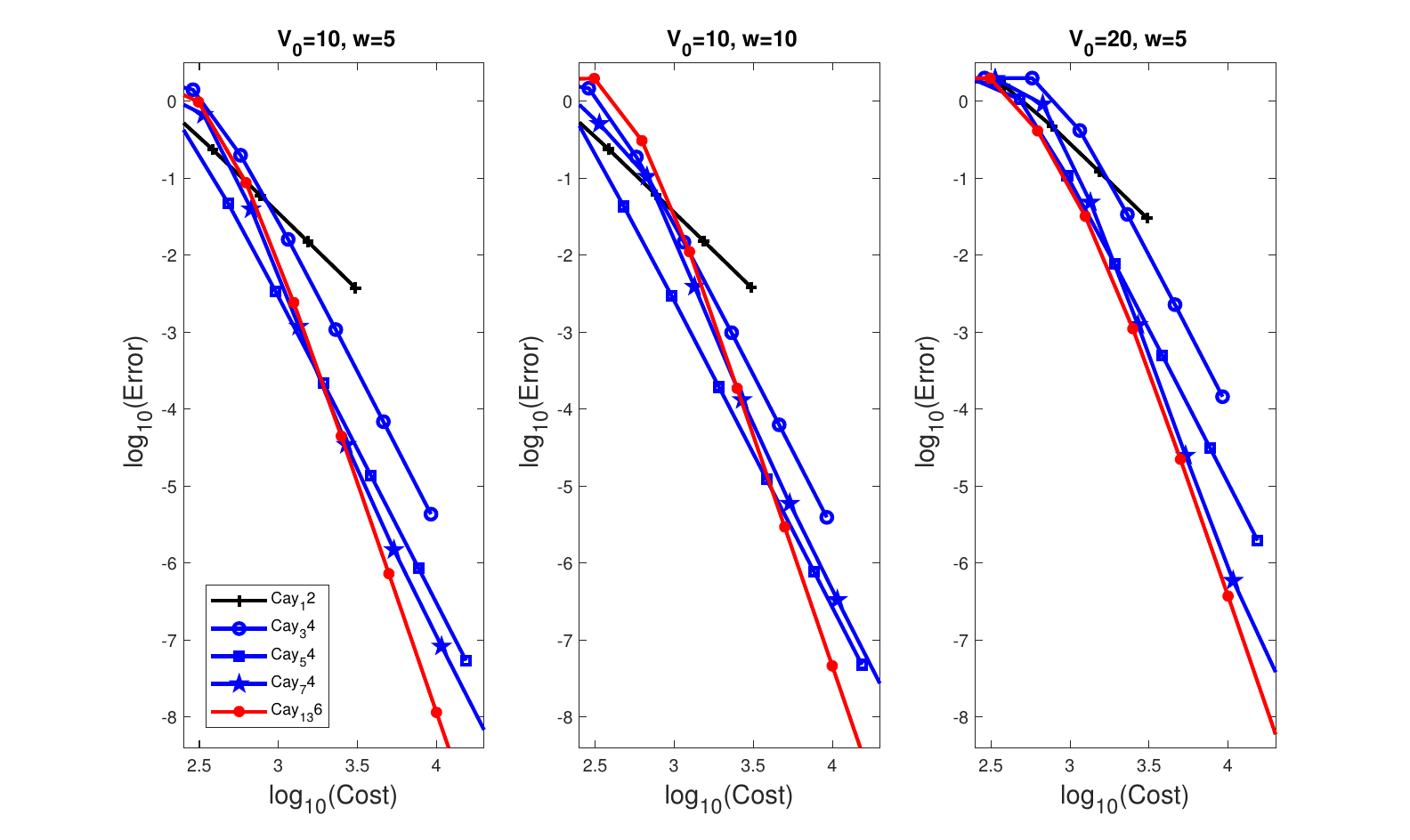}	
	\end{center}
	\caption{{Two-norm error in the fundamental matrix solution of \eqref{r-z.1} at the final time for different $m$-stage modified Cayley--Magnus methods of order $p$, Cay$_mp$, versus the computational cost (measured as the number of Cayley maps) for different choices of $V_0$ and $\omega$.	}}
	\label{fig00}
\end{figure}

We also compare the modified Cayley--Magnus schemes with other classes of Lie group integrators, namely Runge--Kutta--Gauss--Legendre (RKGL) methods and exponential Magnus integrators. In this context, it is important to note that the computational cost of different (exponential and rational) Lie group integrators is highly problem dependent. 

For instance, an RKGL method of order $p = 2s$ requires solving linear systems of dimension $sd \times sd$. For dense matrices, this results in a computational cost approximately $s^3$ times larger. However, for sparse matrices, the additional cost can be significantly higher due to the loss of structure. In such cases, we assume a cost equivalent to $2s^3$ times that of solving a linear system of dimension $k$, which corresponds to the cost of a single Cayley map. Accordingly, their cost is estimated as equivalent to 16 and 54 Cayley maps per step for orders four and six, respectively.

On the other hand, the computational cost of exponential Magnus integrators is also highly problem-dependent and influenced by the desired level of accuracy. The action of the exponentials must be evaluated up to round-off error, or at least to an accuracy higher than the order of the method. 
The cost of computing the action of the exponential $\e^{hA}$ on a vector $v$, i.e., $\e^{hA}v$, depends on the norm $\|hA\|$. While we wish to employ relatively large time steps once low accuracy is required, this leads to an increase in $\|hA\|$, and consequently, to a higher computational cost for evaluating the exponential action. 

For the numerical experiments considered here, we assume that the computational cost of one full step of exponential integrators (including both Magnus and commutator-free methods) is comparable to that of RKGL methods of the same order.

In our tests, we consider the modified Cayley--Magnus methods $\mathrm{Cay}_34$, $\mathrm{Cay}_54$, and $\mathrm{Cay}_74$ (order 4) listed in Table~\ref{tab:cayley_magnus_methods}. These are compared with the fourth-order exponential Magnus method~\eqref{Mor4} (Mag4), the two-exponential commutator-free Magnus scheme~\eqref{cf_or4} (CF4), and the implicit RKGL method of order 4 (denoted by RKGL4). The computational cost of all these Lie group integrators is estimated to be approximately equivalent to 16 Cayley maps, which is approximately twice the cost of $\mathrm{Cay}_{7}4$.

For order 6, we compare the modified Cayley--Magnus method $\mathrm{Cay}_{13}6$ with the exponential Magnus method~\eqref{Mor6} (Mag6), the six-exponential commutator-free Magnus method (CF6) \cite{blanes06fas}, and the implicit RKGL method (RKGL6). In this case, the computational cost of all methods is taken to be equivalent to 54 Cayley maps, which is roughly four times the cost of $\mathrm{Cay}_{13}6$.

In all cases, we employ Gauss--Legendre quadrature rules of order four and six, respectively, corresponding to 2 and 3 evaluations of the time-dependent matrix per step.

The results are presented in Figure~\ref{fig0}. We observe that the optimised Cayley--Magnus methods with 5 and 7 maps 
outperform the exponential integrators and are significantly more efficient than the RKGL4 method. A similar behaviour is observed for the sixth-order Cayley--Magnus method.

\begin{figure}[h!]
	\begin{center}
		\includegraphics[scale=0.62]{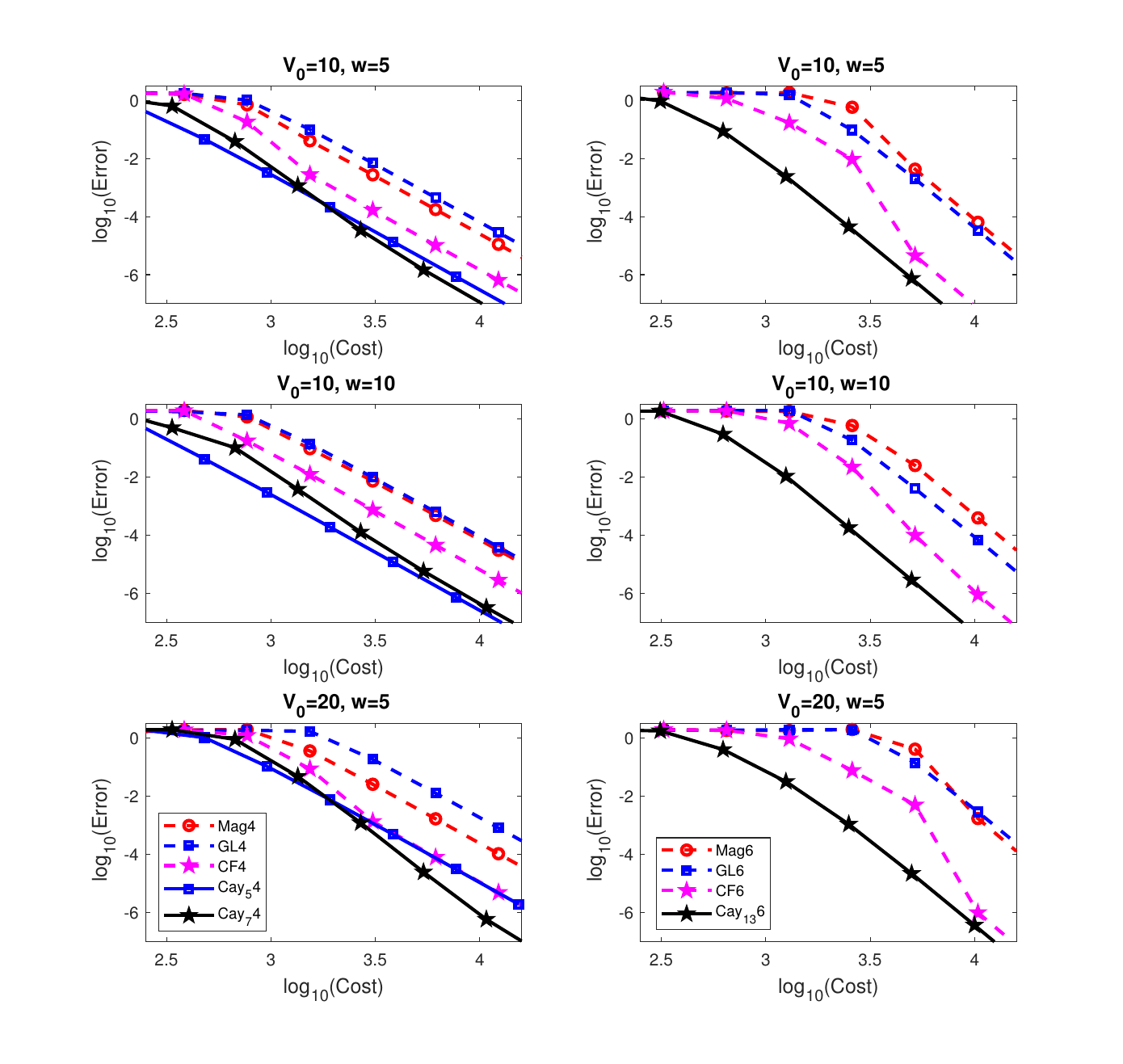}	
	\end{center}
	\caption{{Two-norm error in the fundamental matrix solution of \eqref{r-z.1} at the final time versus the computational cost (number of maps for the modified  Cayley--Magnus methods or the values of time steps multiplied by 16 or 54 for the 4th- and 6th-order RKGL and exponential methods) for different choices of $V_0$ and $\omega$.	}}
	\label{fig0}
\end{figure}



\section*{Acknowledgements}
The work of SB and FC is supported by Ministerio de Ciencia e Innovación (Spain) through project PID2022-
136585NB-C21, \\
MCIN/AEI/10.13039/501100011033/FEDER, UE.  
SB also acknowledges DAMTP at the University of Cambridge
for its hospitatily when part of this work was done and the found from the Ministerio de Ciencia, Innovación y
Universidades for mobility stays in foreign higher education and research centres.

\appendix
\setcounter{equation}{0}
\def\theequation{A.\arabic{equation}}

\section{Exact solution in terms of one-di\-men\-sio\-nal integrals and quadratures}  \label{appendixA}

We present a complete and constructive proof of the fact that the transition matrix ${\cal P}(h)$ in \eqref{eq:SolIt}, associated with the exact solution of \eqref{eq:1}, can be approximated up to order $p=2\ell$ by using only $\ell$ one-dimensional integrals or, alternatively, only the values of the matrix $A$ at the nodes of any quadrature rule of order $p\geq 2\ell$. The proof is developed in several stages, which are detailed below.

\subsection{Legendre polynomials}

We first introduce the shifted Legendre polynomials $P_n(t)$, $n = 0, 1, 2, \ldots$, on the interval $[0,1]$. They are defined by the recurrence relation
\[
P_0(t)=1, \qquad\! P_1(t)=2t-1, \qquad\!
P_{n+1}(t) = \frac{2n+1}{n+1} (2 t -1) P_n(t) - \frac{n}{n+1} P_{n-1}(t),
\]
and satisfy several properties that will be useful in the sequel:

\medskip

\noindent
(1) Orthogonality:
\[
\int_0^1P_m(t)P_n(t)\,dt=\frac1{2n+1} \delta_{m,n}, \quad \mbox{or} \quad \int_0^hP_m(\tfrac{t}{h})P_n(\tfrac{t}{h})\,dt=\frac{h}{2n+1} \delta_{m,n}.
\]

\medskip

\noindent
(2) Integral property for $t\in[0,h]$:
\[
\begin{aligned}
	& \int_0^{t} P_0\big(\tfrac{t_1}{h}\big)\,dt_1 = \frac{h}2 \left[P_0\big(\tfrac{t}{h}\big)+P_1\big(\tfrac{t}{h}\big) \right],\\
	& \int_0^{t} P_n\big(\tfrac{t_1}{h}\big)\,dt_1 = \frac{h}{2(2n+1)} \left[P_{n+1}\big(\tfrac{t}{h}\big)-P_{n-1}\big(\tfrac{t}{h}\big) \right], \qquad n\geq 1,
\end{aligned}
\]
so that integration shifts the polynomial index by one unit.

\medskip

\noindent
(3) Product of Legendre polynomials (Neumann--Adams formula),
\begin{displaymath}
	P_m(t)P_n(t)=\sum_{r=0}^{\min(m,n)} \sigma_r^{(m,n)} P_{m+n-2r}(t),
\end{displaymath}
where the coefficients $\sigma_r^{(m,n)}$ are rational numbers \cite{alsalam56opt}. Thus, the product can be written as a linear combination of Legendre polynomials $P_k(t)$ whose index $k$ ranges from $|m-n|$ to $m+n$. In particular,
\[
\begin{aligned}
	& P_m(t)P_0(t)=\sigma_0^{(m,0)} P_{m}(t), \quad\\
	& P_m(t)P_1(t)=\sigma_0^{(m,1)} P_{m+1}(t)+\sigma_1^{(m,1)} P_{m-1}(t), \qquad m\geq 1,\\
	& P_m(t)P_2(t)=\sigma_0^{(m,2)} P_{m+2}(t)+\sigma_1^{(m,2)} P_{m}(t)+\sigma_2^{(m,2)} P_{m-2}(t), \qquad m\geq 2.
\end{aligned}
\]

\medskip

\noindent
(4) As a consequence of the previous properties, one has
\begin{displaymath}
	\int_0^h P_{m}\big(\tfrac{t_1}{h}\big)\,dt_1\int_0^{t_1} P_{n}\big(\tfrac{t_2}{h}\big)\,dt_2 =0,
	\qquad \mbox{if} \quad |m-n|\geq 2,
\end{displaymath}
and, more generally,
\begin{eqnarray}
	&&\int_0^h P_{\ell_1}\big(\tfrac{t_1}{h}\big)\,dt_1\int_0^{t_1} P_{\ell_2}\big(\tfrac{t_2}{h}\big)\,dt_2\cdots
	\int_0^{t_{k-1}} P_{\ell_k}\big(\tfrac{t_k}{h}\big)\,dt_k =0
	\label{eq:Int_k0}  \\
	&& 
	\qquad \mbox{if} \quad \max\{\ell_1,\ldots,\ell_k\}- \min\{\ell_1,\ldots,\ell_k\}\geq k.
	\nonumber
\end{eqnarray}

\medskip

\noindent
(5) The analytic matrix function $A(t)$ can be expanded in terms of the shifted Legendre polynomials as in \cite{alvermann11hoc},
\[
A(t)=\frac1{h}\sum_{n=0}^{\infty} \beta_{n+1}(h)  \, P_n\big(\tfrac{t}{h}\big),
\qquad \mbox{with} \qquad
\beta_{n+1}(h)=(2n+1)\int_0^hP_n\big(\tfrac{t}{h}\big) A(t)\,dt.
\]
(The factor $\frac1{h}$ in the expansion of $A(t)$ and the shift by one unit in the index of $\beta_{n+1}(h)$ are introduced solely for convenience.) By orthogonality, it follows that
\[
\beta_{n+1}(h)={\cal O}(h^{n+1}),  \qquad n=0,1,\ldots .
\]
Indeed, if one considers the expansion $A(t) = \sum_{m \ge 0} a_m t^m$, then, since $P_n(t)$ is orthogonal to all monomials $t^m$ with $m < n$, and
\[
\beta_n(h) = (2n-1) h \int_0^1 P_{n-1}(s) A(h s)\,ds,
\]
only terms in $h^m$ with $m \ge n$ appear in $\beta_n(h)$.

\medskip

\noindent
(6) The shifted Legendre polynomials are symmetric with respect to $t=1/2$, namely
$P_n(1-t) = (-1)^n P_n(t)$. Consequently, for suitable coefficients,
\[
P_{2n}\big(\tfrac{t}{h}\big)=\sum_{k=0}^{n}a_{n,k}\big(t-\tfrac{h}2\big)^{2k}, \qquad
P_{2n+1}\big(\tfrac{t}{h}\big)=\sum_{k=0}^{n}b_{n,k}\big(t-\tfrac{h}2\big)^{2k+1}.
\]
Moreover, by considering the Taylor expansion of $A(t)$, $t \in [0,h]$, about the midpoint $h/2$, one obtains
\begin{eqnarray*}
	\delta_k(h)&=&\int_0^h\big(t-\tfrac{h}2\big)^k A(t)\,dt=
	\int_{-h/2}^{h/2} \tau^k A\big(\tau+\tfrac{h}2\big)\,d\tau \\
	&=& 
	\sum_{n=0}^{\infty} \frac1{n!}\left.\frac{d^n A(\tau+h/2)}{d\tau^n}\right|_{\tau=0}\int_{-h/2}^{h/2} \tau^{k+n}\,d\tau,
\end{eqnarray*}
so that $\delta_k(-h) = (-1)^k \delta_k(h)$. As a consequence,
\[
\beta_{2k-1}(-h)=-\beta_{2k-1}(h), \qquad
\beta_{2k}(-h)=\beta_{2k}(h), \qquad k=1,2,\ldots
\]

\subsection{The iterated integral $I_k(h)$}

With the aid of the preceding properties, we are now in a position to compute the iterated integrals \eqref{eq:ii_A} appearing in the expansion
${\cal P}(h)=\sum_{k=0}^{\infty}I_k(h)$. Clearly,
\[
I_1(h) = \int_0^h A(t)\,dt= \beta_1,
\]
whereas, for $I_2(h)$, a direct computation gives
\allowdisplaybreaks
\begin{align*}
&	I_2(h)=	
	\int_0^h A(t_1)\,dt_1	\int_0^{t_1} A(t_2)\,dt_2 \\
&	=	
	\int_0^h \frac1{h}\sum_{k=0}^{\infty} \beta_{k+1} P_k(\tfrac{t_1}{h})\,dt_1	\int_0^{t_1} \frac1{h}\sum_{n=0}^{\infty} \beta_{n+1} P_n(\tfrac{t_2}{h})\,dt_2 \\
&       =	
	\frac1{h^2}\sum_{k=0}^{\infty} \beta_{k+1} \beta_{1} 
	\int_0^{h} P_k(\tfrac{t_1}{h})\,dt_1	\int_0^{t_1}  P_0(\tfrac{t_2}{h})\,dt_2 \\
 & +	
	\frac1{h^2}\sum_{k=0}^{\infty}\sum_{n=1}^{\infty} \beta_{k+1} \beta_{n+1} 
	\int_0^{h} P_k(\tfrac{t_1}{h})\,dt_1	\int_0^{t_1}  P_n(\tfrac{t_2}{h})\,dt_2 \\
 & =	
	\frac1{h}\sum_{k=0}^{\infty} \beta_{k+1} \beta_{1} 
	\int_0^{h} P_k(\tfrac{t_1}{h})\,dt_1	\frac12\left(P_{1}(\tfrac{t_1}{h})+P_{0}(\tfrac{t_1}{h})\right) \\
 & +	
	\frac1{h}\sum_{k=0}^{\infty}\sum_{n=1}^{\infty} \beta_{k+1} \beta_{n+1} \frac{1}{2(2n+1)}
	\int_0^{h} P_k(\tfrac{t_1}{h})\,dt_1	\left(P_{n+1}(\tfrac{t_1}{h})-P_{n-1}(\tfrac{t_1}{h})\right) \\
& =	\ldots  \\
 & =	\frac12 \beta_{1}\beta_{1}+ 
	\sum_{n=0}^{\infty}	\frac{1}{2(2n+1)(2n+3)} \,  [\beta_{n+2}, \beta_{n+1}]  \\
& =	\frac12 \beta_{1}\beta_{1}+  \frac16 [\beta_{2},\beta_{1}]+  \frac1{30} [\beta_{3},\beta_{2}]+  \frac1{70} [\beta_{4},\beta_{3}]+
	\cdots
\end{align*}
The same procedure can be carried out in general, leading to
\begin{equation} \label{eq:Int_k}
	\begin{aligned}
		I_k(h) & =\int_0^h A(t_1)\,dt_1\int_0^{t_1} A(t_2)\,dt_2\cdots
		\int_0^{t_{k-1}} A(t_k)\,dt_k \\
		&=	\sum_{n=k}^{\infty} \sum_{\substack{ \ell_1,\ldots,\ell_k \geq 1\\ \ell_1+\cdots+\ell_k=n }}
		\sigma_{\ell_1,\ldots,\ell_k} \, \beta_{\ell_1}(h) \ldots\beta_{\ell_k}(h),
	\end{aligned}
\end{equation}
where, by virtue of \eqref{eq:Int_k0},
\begin{equation} \label{eq:sigma}
	\sigma_{\ell_1,\ldots,\ell_k}=0 
	\qquad \mbox{if} \quad \max(\ell_1,\ldots,\ell_k)- \min(\ell_1,\ldots,\ell_k)\geq k.
\end{equation}

\subsection{Quadrature rules}

Let $b_i$ and $c_i$, $i=1,\ldots,\mu$, be the weights and nodes of a quadrature rule of order $p$. Then, for the analytic matrix function $A(t)$, one has
\[
\int_0^h A(t)\,dt = h\sum_{i=1}^{\mu} b_i \, A(c_i h) + {\cal O}(h^{p+1}).
\]
In the particular case of the Gauss--Legendre (GL) quadrature rule, $\mu = \ell$ and $p= 2 \ell$. When this quadrature rule is used to approximate the integrals
$\beta_{n+1}(h)$, one obtains
\begin{eqnarray}
	\beta_{n+1}(h)&\!\!\!=\!\!\!&(2n+1)\int_0^hP_n \big(\tfrac{t}{h}\big) A(t)\,dt =
	\frac{2n+1}{h^n}\int_0^h h^nP_n\big(\tfrac{t}{h}\big) A(t)\,dt \nonumber \\
	& \!\!\!=\!\!\! & 
	\frac{2n+1}{h^n} \left[h\sum_{i=1}^{\ell} b_i h^n P_n(c_i) A(c_i h)+ {\cal O}(h^{2\ell+1})\right] \nonumber  \\
	& \!\!\!=\!\!\! & 
	(2n+1) h\sum_{i=1}^{\ell} b_i P_n(c_i) A(c_i h)+ {\cal O}(h^{2\ell+1-n}),
	\label{eq:Int_Num_p}
\end{eqnarray}
so that the approximation is just of order $2\ell-n$, $n=0,1,\ldots,\ell-1$. Defining
\begin{equation} \label{eq:alpha_n}
	\beta_{n+1}^{(\ell)}(h) \equiv  (2n+1) h\sum_{i=1}^{\ell} b_i P_n(c_i) A(c_i h),
\end{equation}
 clearly
\[
\beta_{n+1}^{(\ell)}(h)={\cal O}(h^{n+1}), \qquad\quad 
\beta_{n+1}^{(\ell)}(h)=\beta_{n+1}(h)+ {\cal O}(h^{2\ell+1-n}),
\]
and
\[
\beta_{2k-1}^{(\ell)}(-h)=-\beta_{2k-1}^{(\ell)}(h), \qquad\quad
\beta_{2k}^{(\ell)}(-h)=\beta_{2k}^{(\ell)}(h), \qquad k=1,2,\ldots.
\]
Similar results are obtained for an arbitrary symmetric quadrature rule of order $p\geq 2\ell$.

\subsection{Proof of the main result}

We now have all the ingredients required to prove our claim that ${\cal P}(h)$ can be approximated up to order ${\cal O}(h^{2\ell})$ by using only $\ell$ univariate integrals or any quadrature rule of order $p\geq 2\ell$.

Let us define
\[
{\cal P}^{[2\ell]}(h) \equiv \sum_{k=0}^{2\ell}I_k(h).
\]
Since $I_k(h)={\cal O}(h^k)$, it follows that
\[
{\cal P}^{[2\ell]}(h)={\cal P}(h)+{\cal O}(h^{2\ell+1}).
\]
Taking \eqref{eq:Int_k} into account, we obtain
\begin{displaymath} 
	{\cal P}^{[2\ell]}(h)=I+\sum_{k=1}^{2\ell}\sum_{n=k}^{\infty} \sum_{\substack{ \ell_1,\ldots,\ell_k\geq1 \\ \ell_1+\cdots+\ell_k=n }}
	\sigma_{\ell_1,\ldots,\ell_k} \, \beta_{\ell_1}(h)\ldots\beta_{\ell_k}(h),
\end{displaymath}
featuring the full sequence of matrices $\{\beta_k(h)\}_{k=1}^{\infty}$.

Since $\beta_{\ell_m}(h) = \mathcal{O}(h^{\ell_m})$, $m=1, 2, \ldots$, one has
\[
\beta_{\ell_1}(h)\cdots\beta_{\ell_k}(h) = {\cal O}(h^{\mu}), \qquad \mbox{with} \qquad \mu=\sum_{i=1}^k \ell_i,
\]
and also
\[
\beta_{\ell_1}(h)\cdots\beta_{\ell_k}(h)  = \beta_{\ell_m}(h) \, {\cal O}(h^{\mu_{m}}), \qquad \mu_{m}=\sum_{i\neq m} \ell_i, \qquad m=1,\ldots, k.
\]
In other words, each factor $\beta_{\ell_m}(h)$, $m=1,\ldots,k$, in \eqref{eq:Int_k} is multiplied by terms whose total contribution is of order ${\cal O}(h^{\mu_{m}})$. The next proposition gives further information on the value of $\mu_m$.

\begin{proposition} \label{prop_A2}
	Each product $\beta_{\ell_1}(h) \cdots \beta_{\ell_k}(h)$ in the expression of $I_k(h)$ given in \eqref{eq:Int_k} satisfies
	\[
	\beta_{\ell_1}(h) \cdots\beta_{\ell_k}(h) = \beta_{\ell_m}(h) \, {\cal O}(h^{\mu_{m}}), \qquad \mbox{with} \qquad \mu_{m}\geq \ell_m-1
	\]
	for every $m=1,\ldots,k$.
\end{proposition}

\begin{proof}
	This is clearly true for $k=1,2$, as shown by the explicit computations above. Furthermore,
	\[
	\mu_m \geq \min_{i\neq m}\{\ell_i\}(k-1),
	\]
	since it contains $k-1$ factors, each of order at least $\min_{i\neq m}\{\ell_i\}$. From \eqref{eq:sigma}, it follows that
	\begin{equation} \label{eq:min_ell}
		\min_{i\neq m}\{\ell_i\} \geq  \max\{ 1,\max_{j}\{\ell_j\}-k+1\} \geq \max\{ 1,\ell_m-k+1\},
	\end{equation}
	because $\max_{j}\{\ell_j\}\geq \ell_m$ and $\min_{i\neq m}\{\ell_i\}\geq 1$.
	Therefore:
	\begin{itemize}
		\item If $\ell_m\leq k$ (with $k \ge 2$), then $\min_{i\neq m}\{\ell_i\}\geq 1$, and hence
		\[
		\mu_m \geq \min_{i\neq m}\{\ell_i\}(k-1) \geq \ell_m-1,
		\]
		as required.
		
		\item If $\ell_m> k$, then, by \eqref{eq:min_ell},
		\[
		\mu_{m} \geq \min_{i\neq m}\{\ell_i\}(k-1)\geq (\ell_m-k+1)(k-1).
		\]
		The right-hand side defines a parabola in the variable $k$, with extreme values attained at the boundary points $k=2$ and $k=\ell_m-1$ (for $\ell_m\geq 3$). In both cases one obtains $\mu_{m}\geq \ell_m-1$.
	\end{itemize}
	This completes the proof.
\end{proof}

\begin{proposition} \label{prop_A1}
	Let $J(\beta_1,\ldots,\beta_{\ell};h)$ denote the function obtained from $\mathcal{P}^{[2 \ell]}(h)$ by setting $\beta_n=0$ for all $n>\ell$, and 
	$\beta_{\ell_1}\cdots\beta_{\ell_k}= 0$ if $\sum_{i=1}^k \ell_i>2\ell$, i.e.
	\begin{equation} \label{expre_J}
		J(\beta_1,\ldots,\beta_{\ell};h)= I + \sum_{k=1}^{2 \ell} \sum_{\substack{ \ell_1,\ldots,\ell_k\leq\ell \\ \ell_1+\cdots+\ell_k \le 2 \ell }}
		\sigma_{\ell_1,\ldots,\ell_k} \, \beta_{\ell_1}(h) \cdots\beta_{\ell_k}(h).
	\end{equation}
	Then
	\[
	J(\beta_1,\ldots,\beta_{\ell};h)= \mathcal{P}(h)+{\cal O}(h^{2\ell+1}).
	\]
\end{proposition}

\begin{proof}
	It is enough to prove that
	\[
	J(\beta_1,\ldots,\beta_{\ell};h)= {\cal P}^{[2\ell]}(h)+{\cal O}(h^{2\ell+1}).
	\]
	From Proposition \ref{prop_A2} we conclude that all terms in ${\cal P}^{[2\ell]}(h)$ satisfy
	\begin{equation} \label{eq:prod_beta}
		\beta_{\ell_1}\cdots\beta_{\ell_k} = \beta_{\ell_m} \, {\cal O}(h^{\mu_{m}})= {\cal O}(h^{\ell_m+\mu_{m}}) , \qquad \mu_{m}\geq \ell_m-1, 
		\qquad m=1,2,\ldots,k.
	\end{equation}
	Consequently, if $\ell_m>\ell$ for some $m$, then $\ell_m+\mu_{m}>2\ell$, and therefore such a term contributes only at order $2 \ell +1$ or higher. Obviously, this is also the case if $\ell_1+\ldots+\ell_k > 2\ell$.
\end{proof}

Proposition \ref{prop_A1} shows that an approximation to $\mathcal{P}^{[2\ell]}(h)$ up to order $h^{2\ell}$ can be obtained by using solely the $\ell$ univariate integrals $\beta_1(h), \ldots, \beta_{\ell}(h)$ through \eqref{expre_J}. 
For instance, for order four we have  ${\cal P}^{[4]}(h)=I+I_1(h)+I_2(h)+I_3(h)+I_4(h)$, where $I_1(h),I_2(h)$ have already been computed, and a simple computation confirms that
\[
J(\beta_1,\beta_2;h)= I+\beta_1+\frac12 \beta_1^2+\frac16 [\beta_2,\beta_1]+\frac16 \beta_1^3+\frac1{12} \left(\beta_1[\beta_2,\beta_1]+[\beta_2,\beta_1]\beta_1\right)+\frac1{24} \beta_1^4
\]
and $J(\beta_1,\beta_2;h)={\cal P}(h)+{\cal O}(h^5)$, being a 4th-order approximation in terms of just two one-dimensional integrals, $\beta_1$ and $\beta_2$.

It remains to determine the accuracy with which these integrals must be approximated in order to retain the overall order $2\ell$. The following proposition provides the answer.

\begin{proposition}\label{prop_A4}
	Let $\beta_{n+1}^{(\ell)}(h)$ be any analytical or numerical approximation of the univariate integral $\beta_{n+1}(h)$ of order $2\ell-n$, that is,
	\[
	\beta_{n+1}^{(\ell)}(h)=\beta_{n+1}(h)+{\cal O}(h^{2\ell-n+1}), \qquad n \ge 0.
	\]
	Then
	\[
	J(\beta_1^{(\ell)},\ldots,\beta_{\ell}^{(\ell)};h)=  \mathcal{P}(h)+{\cal O}(h^{2\ell+1}).
	\]
\end{proposition}

\begin{proof}
	Since $\beta_{n+1}^{(\ell)}(h)=\beta_{n+1}(h)+{\cal O}(h^{\nu})$ with $\nu\geq  2\ell+1-n$, it follows that
	$\beta_{\ell_1}(h)=\beta_{\ell_1}^{(\ell)}(h)+{\cal O}(h^{2\ell+2-\ell_1})$. Consequently,
	\[
	\begin{aligned}
		\beta_{\ell_1}\ldots\beta_{\ell_k}  & =
		\Big(\beta_{\ell_1}^{(\ell)}+{\cal O}(h^{2\ell+2-\ell_1})\Big)\beta_{\ell_2}\ldots\beta_{\ell_k} \\
		& =
		\beta_{\ell_1}^{(\ell)}\beta_{\ell_2}\ldots\beta_{\ell_k}+
		{\cal O}(h^{2\ell+2-\ell_1}) {\cal O}(h^{\ell_1-1}) \\
		&=
		\beta_{\ell_1}^{(\ell)}\beta_{\ell_2}\ldots\beta_{\ell_k} + {\cal O}(h^{2\ell+1}),
	\end{aligned}
	\]
	since $\beta_{\ell_2}\ldots\beta_{\ell_k}= {\cal O}(h^{\ell_1-1})$. This follows from the fact that $\beta_{\ell_1} = \mathcal{O}(h^{\ell_1})$ and, according to \eqref{eq:prod_beta}, $\beta_{\ell_2}\cdots\beta_{\ell_k} = \mathcal{O}(h^{\ell_1 -1})$. The same argument can then be applied successively to $\beta_{\ell_2},\ldots,\beta_{\ell_k}$, yielding finally
	\[
	\beta_{\ell_1}\cdots\beta_{\ell_k} =
	\beta_{\ell_1}^{(\ell)}\cdots\beta_{\ell_k}^{(\ell)}+ {\cal O}(h^{2\ell+1})
	\]
	for every term in the series defining $J(\beta_1,\ldots,\beta_{\ell};h)$.
\end{proof}

\begin{corollary} \label{corol_A5}
	Let $b_i,c_i$, $i=1,\ldots,\mu$, be the weights and nodes of an arbitrary quadrature rule of order $p\geq2\ell$, and define, as in \eqref{eq:alpha_n},
	\[
	\beta_{n+1}^{(\ell)}(h) = (2n+1) h\sum_{i=1}^{\mu} b_i P_n(c_i) A(c_i h).
	\]
	Then
	\[
	J(\beta_1^{(\ell)},\ldots,\beta_{\ell}^{(\ell)};h)={\cal P}(h)+{\cal O}(h^{2\ell+1}).
	\]
\end{corollary}

\begin{proof}
	From \eqref{eq:Int_Num_p} it is clear that
	\[
	\beta_{n+1}^{(\ell)}(h) =\beta_{n+1}(h)+{\cal O}(h^{2\ell+1-n}), \qquad n=0,1,\ldots,\ell-1,
	\]
	and Proposition \ref{prop_A4} therefore applies.
\end{proof}

If the GL quadrature rule with $\ell$ nodes is used to approximate the univariate integrals $\beta_{n+1}(h)$, then $\mu=\ell$, $p=2\ell$, and one still obtains an approximation to $\mathcal{P}(h)$ of order $2\ell$.

For instance, using the 4th-order GL quadrature rule with $c_1=\frac12-\frac{\sqrt{3}}{6}, \ c_2=\frac12+\frac{\sqrt{3}}{6}, \ b_1=b_2=\frac12, \ A_i=A(c_ih), \ i=1,2,$ we have
\[
\begin{aligned}
	\beta_{1}^{(2)}  &= h\Big(b_1P_0(c_1)A_1+P_0(c_2)A_2\Big)=\frac{h}2(A_1+A_2) = \beta_1+{\cal O}(h^5) \\
	\beta_{2}^{(2)}  &= h\Big(b_1P_1(c_1)A_1+P_1(c_2)A_2\Big)=\frac{\sqrt{3}h}2(A_2-A_1) = \beta_2+{\cal O}(h^4) ,
\end{aligned}
\]
and $J(\beta_1^{(2)},\beta_2^{(2)};h)={\cal P}(h)+{\cal O}(h^5)$.

\paragraph{Remark.}
We have shown that the fundamental matrix solution ${\cal P}(h)$ can be approximated up to order ${\cal O}(h^{2\ell})$ by means of the one-dimensional integrals
\[
\beta_{n+1}(h)=(2n+1)\int_0^h P_n\!\left(\frac{t}{h}\right) A(t)\,dt, \qquad n=0,1,\ldots,\ell-1.
\]
The actual approximation can be constructed as follows. We consider the expansion of the matrix $A(t)$ in terms of shifted Legendre polynomials,
\[
A(t)=\frac{1}{h}\sum_{n=0}^{\infty}   \beta_{n+1}(h) P_n\!\left(\frac{t}{h}\right),
\]
and substitute it into the nested integrals $I_n(h)$, for $n=1,\ldots,2\ell$. By carrying out the resulting integrations analytically, retaining only those terms involving $\beta_1,\ldots,\beta_{\ell}$, and truncating the expansion beyond order $2\ell$, one obtains the desired approximation.
Moreover, extending the expansion to include, for instance, $I_{2\ell+1}, I_{2\ell+2}$, the terms $\beta_1,\ldots,\beta_{\ell+1}$ and truncating the solution at order $2\ell+1$ or $2\ell+2$ provides additional insight into the approximation error, since the leading neglected terms can then be explicitly identified. This can be useful when constructing
numerical integration methods.

\subsection{Approximating the solution in different bases}

In some cases, it is convenient to express approximations of the exact solution using bases other than the univariate integrals $\beta_{n+1}(h)$ or their approximations $\beta_{n+1}^{(\ell)}(h)$. In particular, alternative representations that relate the error terms to derivatives of the matrix $A(t)$ are especially appropriate, since these derivatives are often known in advance.

One such basis is formed by the momentum matrices \cite{blanes00iho}
\begin{equation} \label{def_Bn}
	B^{(n)}(h)
	=\frac1{h^{n+1}}\int_0^h\left(t-\tfrac{h}{2}\right)^n A(t)dt
	=\frac1{h^{n+1}}\int_{-h/2}^{h/2}t^n A(t+\tfrac{h}2)dt, \quad n=0,1,\ldots, 
\end{equation}
since they can be written in terms of the integrals $\beta_n(h)$, and the momentum integrals for $n=0,1,\ldots,\ell-1$ suffice to approximate the solution up to order $2\ell$. By construction, the $B^{(2k)}$s (respectively, $B^{(2k+1)}$s) are odd (respect., even)
functions of $h$, although $B^{(n)}\neq{\cal O}(h^{n})$ for $n>1$. For this reason it is more convenient to take appropriate linear combinations satisfying this last requirement. This can be achieved by considering the Taylor expansion of $A(t)$ about the midpoint
\[
A\!\left(t+\frac{h}{2}\right)= \frac{1}{h}\sum_{k=0}^{\infty}\alpha_{k+1}(h)\left(\frac{t}{h}\right)^{k}, \qquad
\alpha_{k+1}(h)=h^{k+1}\left.\frac{1}{k!}\frac{d^kA(t)}{dt^k}\right|_{t=\frac{h}{2}},
\]
where now $\alpha_{n}(h)={\cal O}(h^{n})$,
\begin{equation} \label{expan_A}
	\alpha_{2k-1}(-h)=-\alpha_{2k-1}(h), \qquad
	\alpha_{2k}(-h)=\alpha_{2k}(h), \qquad k=1,2,\ldots,
\end{equation}
and by inserting \eqref{expan_A} into \eqref{def_Bn}, we obtain
\begin{displaymath}
	B^{(n)}(h) =\frac{1}{h}\sum_{k=0}^{\infty}\frac{1-(-1)^{k+n+1}}{2^{k+n+1}}\, \alpha_{k+1}(h), \qquad n=0,1,\ldots.
\end{displaymath}
To construct an approximation of ${\cal P}(h)$ up to order $2\ell$ in terms of the matrices $B^{(0)},\ldots,B^{(\ell-1)}$, the strategy is to first express ${\cal P}(h)$ in terms of the coefficients $\alpha_{k+1}(h)$ and then retain only those contributions involving $\alpha_1,\ldots,\alpha_{\ell}$ up to order $2\ell$. Next, with  the linear combinations
\[
\widetilde B^{(n)}(h)
=\frac{1}{h}\sum_{k=0}^{\ell-1} T_{n,k}\,\alpha_{k+1}(h), 
\qquad \text{with} \qquad
T_{n,k}=\frac{1-(-1)^{k+n+1}}{(k+n+1)2^{k+n+1}},
\]
the approximation is written in terms of $\widetilde B^{(n)}(h)$ and, subsequently, these quantities are replaced by $B^{(n)}(h)$. Since the matrix with entries $T_{n,k}$ is invertible, the conditions imposed on $B^{(n)}(h)$ and on $\alpha_{n+1}(h)$, for $n=0,1,\ldots,\ell-1$, are equivalent up to order $2\ell$.

For example, to order four it suffices to consider
\[
 A\!\left(t+\frac{h}{2}\right)= \frac{1}{h}\left(\alpha_{1}+\alpha_{2}\frac{t}{h}\right), 
\]
and  ${\cal P}^{[4]}(h)=I+I_1(h)+I_2(h)+I_3(h)+I_4(h)$, with
\begin{eqnarray*}
	I_1(h)&=&\int_0^hA(t)dt=	\int_{-\frac{h}2}^{\frac{h}2}\frac{1}{h}\left(\alpha_{1}+\alpha_{2}\frac{t}{h}\right)dt=	\alpha_1 \\
	I_2(h)&=&	\int_{-\frac{h}2}^{\frac{h}2}\frac{1}{h}\left(\alpha_{1}+\alpha_{2}\frac{t_1}{h}\right)dt_1
	\int_{-\frac{h}2}^{t_1}\frac{1}{h}\left(\alpha_{1}+\alpha_{2}\frac{t_2}{h}\right)dt_2=\frac12 \alpha_1^2+\frac1{12}[\alpha_2,\alpha_1] \nonumber\\
	I_3(h)&=&	
	\frac16 \alpha_1^3+ \frac1{24}(\alpha_1[\alpha_2,\alpha_1]+[\alpha_2,\alpha_1]\alpha_1)+ {\cal O}(h^5)  \nonumber\\
	I_4(h) & = & \frac1{24} \alpha_1^4+ {\cal O}(h^5).
\end{eqnarray*}
We can then take
\[
J(\alpha_1,\alpha_2;h)=\alpha_1 + \frac12 \alpha_1^2+\frac1{12}[\alpha_2,\alpha_1] + \frac16 \alpha_1^3+ \frac1{24}(\alpha_1[\alpha_2,\alpha_1]+[\alpha_2,\alpha_1]\alpha_1)+\frac1{24} \alpha_1^4,
\]
and replace $\alpha_1,\alpha_2$ by
\begin{equation*} 
\begin{aligned}
  & 	\alpha_1^{(4)} = \frac{h}2(A_1+A_2) =hB^{(0)}(h)+{\cal O}(h^5), \\
  &	\alpha_2^{(4)} = h\sqrt{3}(A_2-A_1) =12 hB^{(1)}(h)+{\cal O}(h^4). 
 \end{aligned} 
\end{equation*}

If the Gauss--Legendre quadrature rule is used, then for $\ell=1,2,3$ we have the following values for $\alpha_{n+1}^{(2\ell)}, \ n+1\leq\ell$ (other quadrature rules can  be used just as well)
\begin{itemize}
	\item $\ell = 1$, order two: $c_1=\frac12$ and $\widetilde B^{(0)}(h)=\frac1{h}\alpha_1$. Replacing $\widetilde B^{(0)}(h)$ by $B^{(0)}(h)$ 
	and approximating the integral by the quadrature rule we have
	\begin{equation} \label{g_or2}
		\alpha_1^{(2)} = hA_1.
	\end{equation}    
	\item $\ell = 2$, order four: $c_1=\frac12-\frac{\sqrt{3}}{6}, \ c_2=\frac12+\frac{\sqrt{3}}{6}$, and similarly
	\[
	\widetilde B^{(0)}=\frac1{h}\alpha_1, \qquad 
	\widetilde B^{(1)}=\frac1{12h}\alpha_2, \qquad 
	\] 
	so that	
	\begin{equation} \label{g_or4} 
		\alpha_1^{(4)} =\frac{h}2(A_1+A_2), \qquad \qquad \alpha_2^{(4)}=h^2\frac{A_2-A_1}{(c_2-c_1)h}=h\sqrt{3}(A_2-A_1).
	\end{equation} 
	\item $\ell = 3$, order six: $c_1=\frac12-\sqrt{\frac3{20}}, \ c_2=\frac12, \ c_3=\frac12+\sqrt{\frac3{20}}$, and 
	\begin{displaymath}
	\widetilde B^{(0)}=\frac1{h}(\alpha_1+\tfrac1{12}\alpha_3), \qquad 
	\widetilde B^{(1)}(h)=\frac1{12h}\alpha_2, \qquad 
	\widetilde B^{(2)}=\frac1{h}(\tfrac1{12}\alpha_1+\tfrac1{80}\alpha_3), \qquad 
	\end{displaymath}
	so that
	\[
	\alpha_1=h(\tfrac9{4}\widetilde B^{(0)}-15\widetilde B^{(2)}), \quad
	\alpha_2=12h\widetilde B^{(1)}, \quad
	\alpha_3=h(-15\widetilde B^{(0)}+180\widetilde B^{(2)}), 
	\] 
	and finally
	\begin{equation} \label{g_or6}
		\begin{aligned}
			& \alpha_1^{(6)} =hA_2, \qquad \qquad \alpha_2^{(6)} =h^2\frac{A_3-A_1}{(c_3-c_1)h}=\frac{\sqrt{15}h}3(A_3-A_1), \\
			& \alpha_3^{(6)} =\frac{h^3}{2}\frac{A_1-2A_2+A_3}{((c_3-c_2)h)^2}=\frac{10h}3(A_1-2A_2+A_3).
		\end{aligned} 
	\end{equation}
	
\end{itemize}
Observe that, at order six, we have
\[
\begin{array}{l}
	\alpha_1^{(6)} = h\big(\tfrac{9}{4} B^{(0)} - 15 B^{(2)}\big) + \mathcal{O}(h^5), \\[4pt]
	\alpha_2^{(6)} = 12h\, B^{(1)} + \mathcal{O}(h^6), \\[4pt]
	\alpha_3^{(6)} = h\big(-15 B^{(0)} + 180 B^{(2)}\big) + \mathcal{O}(h^5),
\end{array}
\]
which, at first glance, does not appear to exhibit the expected order behaviour,
but it turns out that
\[
\begin{array}{l}
	\alpha_1^{(6)} +\frac1{12}\alpha_3^{(6)} =hB^{(0)}+{\cal O}(h^7), \\
	\frac1{12}\alpha_2^{(6)} =hB^{(1)}(h)+{\cal O}(h^6), 
	\\
	\frac1{12}\alpha_1^{(6)} +\frac1{80}\alpha_3^{(6)} =hB^{(2)}+{\cal O}(h^5),
\end{array}
\] 
as required to provide a 6th-order method.

It is worth to mention that
\[
\beta_{n+1}(h)=\frac{(n!)^2}{(2n)!} \,\alpha_{n+1}(h) + {\cal O}(h^{n+3}),
\]
which shows that the univariate integrals $\beta_{n+1}(h)$ are, to leading order, also proportional to the $n$-th derivative of $A$ evaluated at the midpoint.

\end{document}